\documentclass[12pt]{article}

\usepackage{amssymb}
\usepackage{amsthm}
\usepackage{amsmath}
\usepackage{tikz}

\usepackage{pict2e}

\newtheorem{prop}{Proposition}[section]
\newtheorem{cor}[prop]{Corollary}
\newtheorem{thm}[prop]{Theorem}
\newtheorem{lem}[prop]{Lemma}
\theoremstyle{definition}

\newtheorem{example}[prop]{Example}
\newtheorem{rem}[prop]{Remark}

\def\GAP{\textsf{GAP}}

\def\Irr{{\rm Irr}}
\def\Lin{{\rm Lin}}
  
  \def\F{{\mathbb F}}
\def\GL{\mathrm{GL}}
\def\PSL{\mathrm{PSL}}
\def\PGL{\mathrm{PGL}}
\def\PSU{\mathrm{PSU}}

\def\SL{\mathrm{SL}}

\def\AGL{\mathrm{AGL}}
\def\Qd{\mathrm{Qd}}
\def\SA{\mathrm{SA}}
\def\PGammaL{\mathrm{P}\Gamma\mathrm{L}}
\def\AGammaL{\mathrm{A}\Gamma\mathrm{L}}

\def\dominates{\unrhd}

\renewcommand{\leq}{\leqslant}
\renewcommand{\geq}{\geqslant}

\begin{document}

\title{On the total character of a finite group}
\author{T. Breuer, L. H\'ethelyi, B. K\"ulshammer and M. Sz\H{o}ke}

\date{\today}

\maketitle

\begin{center}
 \textit{Dedicated to the memory of Jon Alperin (1937-2025)}
\end{center}

\abstract{\noindent
The total character $\tau_G$ of a finite group $G$ is the sum of all
irreducible complex characters of $G$, and the total degree of $G$ is $T(G) :=
\tau_G(1)$. A proper subgroup $H$ of $G$ is rich if $\tau_G$ is ``contained''
in the permutation character $(1_H)^G$. 
In the first part of this paper, we investigate rich subgroups whose index is 
a product of two primes. We also consider rich subgroups of symmetric and 
alternating groups. In the second part we establish a formula for $T(G)$ 
in the case where the order of $G$ is a prime power. This result is analogous
to a formula for the class number of $G$ proved by P.~Hall, and it confirms a conjecture by Heffernan and MacHale from 2008. In the last part of the paper, we investigate finite groups $G$ where $T(G)$ is small, in a certain sense.}

\section{Introduction}\label{sect:intro}

Let $H$ be a subgroup of a finite group $G$. The Frobenius graph $\Gamma(G,H)$
is defined as follows: Its vertex set is the disjoint union of $\Irr(G)$ and 
$\Irr(H)$, and $\chi \in \Irr(G)$ and $\psi \in \Irr(H)$ are adjacent if and only if $\psi$ is a constituent of the restriction $\chi_H$
of $\chi$ to $H$. In \cite{partI} the authors investigated the situation where 
$\Gamma(G,H)$ has diameter $3$. If $\Gamma(G,H)$ has diameter 3 and $1 < H < G$
then, by \cite[Proposition~2.2]{partI}, every irreducible character $\chi$ of
$G$ is a constituent of the permutation character $(1_H)^G$. This motivated the 
introduction of the notion of a rich subgroup.

By \cite[Definition~2.7]{partI}, a proper subgroup $H$ of a finite group $G$
is called rich in $G$, if every $\chi \in \Irr(G)$ is a constituent of $(1_H)^G$.
In \cite[Proposition~4.2]{partI},
the authors showed that the index of a nontrivial rich 
subgroup $H$ of $G$ cannot be a prime power.
Moreover, in \cite[Proposition~4.6]{partI},
they showed that a finite group $G$ containing a rich subgroup 
$H$ of index $2p$ where $p$ is a prime, must be a Frobenius group of order 
$p(p+1)$, and $p$ must be a Mersenne prime. 

In Section~\ref{sect:news} of this paper, we will prove an analogous result (Proposition~\ref{prop:indexpr})
for finite groups
containing a rich subgroup whose index is a product of two odd primes. This
result was already announced in \cite{partI}. In 
Section~\ref{sect:symmetric} we will establish
several results on rich subgroups of finite symmetric and alternating groups.

The total character $\tau = \tau_G$ of a finite group $G$ is defined by $\tau 
:= \sum_{\chi \in \Irr(G)} \chi$. Thus $\tau$ is always ``contained'' in the 
regular character $\rho := \rho_G = \sum_{\chi \in \Irr(G)} \chi(1) \chi$
of $G$. More generally, it is always contained in the permutation character 
$(1_H)^G$ whenever $H$ is a rich subgroup of $G$. (Some authors call $\tau_G$
the Gelfand character of $G$; cf. \cite{Soto}, for example.) 

In Section~\ref{sect:degrees} of the present paper, we prove a lower bound for $T(G) :=
\tau_G(1) = \sum_{\chi \in \Irr(G)} \chi(1)$, the degree of the total 
character of $G$, in the case where $G$ has prime power order. This result
(Proposition~\ref{HefMacH}) confirms a conjecture by Heffernan and MacHale in \cite{HM}.
We also establish an analog of P.~Hall's class number formula
\cite[Satz~V.15.2]{HuppertI}) for groups of prime power order (cf. 
Proposition~\ref{TGHall}).

In Section~\ref{sect:classifications} we will classify certain groups
in terms of their total degrees, and in Section~\ref{sect:algorithm}
we will discuss groups of small total degree, in particular the groups
$G$ with $T(G) \leq 100$.

Our notation is mostly standard. For a finite group $G$, we denote by $\Irr(G)$ 
the set of irreducible complex characters of $G$, and by $\Lin(G)$ its subset 
of linear characters. Then $k(G) := |\Irr(G)|$ is the class 
number of $G$. We define $b(G) := \max\{\chi(1): \chi \in \Irr(G)\}$ and 
$T(G) := \sum_{\chi \in \Irr(G)} \chi(1)$. For (virtual) characters $\chi,\psi$
of $G$, we denote by $[\chi,\psi] = [\chi,\psi]_G$ the inner product of $\chi$
and $\psi$. Moreover, we write $\chi_H$ for the restriction of $\chi$ to $H$,
and $\phi^G$ for the induction to $G$ of a (virtual) character $\phi$ of $H$. 
We also set $\Irr(G|\phi) := \{\chi \in \Irr(G): [\chi_H,\phi] \neq 0\}$ and 
$T(G|\phi) := \sum_{\chi \in \Irr(G|\phi)} \chi(1)$.




\section{Rich subgroups of small index}\label{sect:news}

We start with a result on finite simple groups of Lie type which is probably
known to the experts.

\begin{prop}\label{prop:SylowReductive}
Let $S$ be a finite simple group of Lie type in characteristic $r$,
and let $R$ be a Sylow $r$-subgroup of $S$.
Then $R$ is \emph{not} rich in $S$.
\end{prop}

\begin{proof}
Suppose first that $r=2$.
If $S$ is isomorphic to $B_2(2)' \cong A_6$ then 
$T(S) = 46 > 45 = [S:R]$.
Thus the result holds in this case, by~\cite[Corollary~2.8~(i)]{partI}.
If $S$ is isomorphic to $G_2(2)' \cong \PSU(3, 3^2)$ then
$T(S) = 188 < 189 = [S:R]$.
Since $[1_S, (1_R)^S] = 1$ we cannot have $[\chi, (1_R)^S] > 0$ for every
$\chi \in \Irr(S)$.
Thus the result also holds in this case.
If $S$ is isomorphic to ${}^2F_4(2)'$, the Tits simple group,
we have 
$T(S) = 14264 > 8775 = [S:R]$,
and the result holds.

Next suppose that $r=3$ and that $S$ is isomorphic to
${}^2 G_2(3)' \cong \PSL(2,8)$.
Then we have $T(S) = 64 > 56 = [S:R]$.
Thus the result follows as above.

In the following, we may assume that we are not in one of the cases
treated before.
Then $S$ is isomorphic to $O^{r'}(\overline{K}^\sigma)$
where $\overline{K}$ is a simple algebraic group of adjoint type
over the algebraic closure $\overline{\F}_r$ of $\F_r$,
and $\sigma$ is a Steinberg endomorphism of $\overline{K}$
(cf.~\cite[Definition~2.2.8]{GLS3}).
By~\cite[Proposition~2.1]{Kret} and~\cite[Theorem~7.2]{AM},
in combination with Section~13.7 in~\cite{Carter},
the finite group $G:= \overline{K}^\sigma$ has a cuspidal character $\chi$.
In particular, $\chi$ is not contained in the principal series of $G$,
i.~e., $[\chi, (1_R)^G] = 0$.
(Note that $R$ is also a Sylow $r$-subgroup of $G$.)
Thus $R$ is not rich in $G$,
and the result follows from~\cite[Proposition~2.9]{partI}.
\end{proof}

In general, simple groups can have rich Sylow $p$-subgroups,
see~\cite[Tables 3 and 4]{partI}.

Now we apply the result above to rich subgroups whose index is a product of
two odd primes. 

\begin{prop}\label{prop:indexpr}
Let $H$ be a nontrivial rich subgroup of a finite group $G$,
and suppose that $[G:H]$ is a product of two distinct odd primes.
Then $G$ is a Frobenius group of order $p^n r$
where $p$, $n$ and $r = (p^n-1)/(p-1)$ are primes.
Moreover, the Frobenius kernel $M$ of $G$ is elementary abelian of order $p^n$,
and a Frobenius complement $R$ of $G$ acts irreducibly on $M$.
\end{prop}

\begin{proof}
Let $M$ be a minimal normal subgroup of $G$.
We claim that $H$ is contained in $M$.
This is trivial if $M = G$, so we assume $M < G$.
Thus $HM/M$ is a rich subgroup of $G/M$, by \cite[Lemma~2.11]{partI}.
Since $H$ is core-free in $G$, by \cite[Corollary~2.8]{partI},
$M$ is not contained in $H$. Hence $HM/M$ has prime index in $G/M$.
Now~\cite[Proposition~4.2]{partI} implies that $HM/M = 1$.
Thus we have indeed $H < M$. 

Suppose first that $M$ is abelian.
Then $M$ is elementary abelian of order $p^n$ where $p$ is a prime and $n$ is a 
positive integer. Moreover, we have $|H| = p^{n-1}$ and $|G| = p^n r$ where $r 
\neq p$ is a prime.
Also, a Sylow $r$-subgroup $R$ of $G$ acts irreducibly on $M$.
Thus $G$ is a Frobenius group with kernel $M$ and complement $R$.

We claim that $R$ acts transitively on the set of subgroups of order $p^{n-1}$
in $G$. If not, let $K$ be a subgroup of order $p^{n-1}$ in $G$ which is
not conjugate to $H$, and choose $\psi \in \Irr(M)$ with kernel $K$. Then 
we have $\psi^G \in \Irr(G)$ and $(\psi^G)_M = \sum_{g \in R} \psi^g$. 
Moreover, each summand $\psi^g$ has kernel $K^g \neq H$, so that $[(\psi^G)_H,
1_H] = 0$, and we have a contradiction. 

Thus $R$ acts indeed transitively on the set of subgroups of order $p^{n-1}$ 
in $G$. Since $R$ is abelian, the action is regular, and we obtain 
$r = (p^n-1)/(p-1) = \prod_{1 \neq d \mid n} \Phi_d(p)$ where 
$|\Phi_d(p)| > 1$ for each $d$.
Since $r$ is a prime, $n$ has to be a prime as well.
Thus the result follows in this case. 

From now on we may assume that $M$ is nonabelian.
We may even assume that the Fitting subgroup of $G$ is trivial.

Suppose next that $M \neq G$. 
Then $[G:M]$ is a prime $r$, and $[M:H]$ is a prime $p \neq r$.
Since $O_p(M) = 1$ the It\^o-Michler theorem implies that there exists
$\psi \in \Irr(M)$ such that $p$ divides $\psi(1)$.
If $\psi$ is not $G$-stable then $\psi^G \in \Irr(G)$,
and we have the contradiction $p r \leq \psi^G(1) < T(G) \leq [G:H] = p r$.
Thus $\psi$ must be $G$-stable.
But then $\psi$ extends in $r$ ways to an irreducible character of $G$,
and we have the contradiction $p r \leq r \psi(1) < T(G) \leq [G:H] = p r$.

Thus we must have $G = M$, i.~e., $G$ is a nonabelian finite simple group.
We write $[G:H] = pr$ with primes $p$ and $r$, and 
apply the Classification of the Finite Simple Groups.
By \cite[Proposition~4.1]{partI}, we have $|G| < p^2 r^2$, and this
will often suffice for a contradiction.

Suppose first that $G = A_k$ for some $k \geq 5$.
Then the inequality $k!/2 = |A_k| < p^2 r^2 < k^4$ implies that $k \leq 6$.
However, then we have $[G:H] = 15$, and~\cite[Table~1]{partI} gives
a contradiction.

If $G$ is a sporadic simple group or the Tits group then again
the product of the largest two prime divisors of $|G|$
is smaller than $\sqrt{|G|}$.

This leaves the case where $G$ is a simple group of Lie type.
By Proposition~\ref{prop:SylowReductive},
$H$ cannot contain a Sylow $l$-subgroup of $G$
where $l$ is the defining characteristic of $G$.
Thus we have $l \in \{ p, r \}$.
In particular, $l$ is odd;
this excludes the cases $G = {}^2B_2(q)$ and $G = {}^2F_4(q)$, $q = 2^{2m+1}$.

The order of $G$ is a polynomial in $q$, a power of $l$,
divided by some ``small'' number,
and the polynomial in $q$ can be written as the product of
a power of $q$ and factors of the form $q^k \pm 1$,
for some positive integers $k$.

In the following, $N$ will denote an upper bound for $p r$;
we are done if $N < |G|/N$ holds.

If $G \in \{ D_n(q)\; (n \geq 4), F_4(q), E_6(q), E_7(q), E_8(q),
             {}^2D_n(q)\; (n \geq 4), {}^2E_6(q) \}$,
we may choose $N$ as the product of $l$ and
the largest factor $q^k \pm 1$ that occurs.
The same choice is possible for
$G \in \{ A_n(q), B_n(q), C_n(q), {}^2A_n(q) \}$ if $n \geq 3$ holds.
The remaining cases are as follows.

\begin{itemize}
\item
  $G = A_1(q)$.
  Then $|G| = q (q^2-1) / 2$.
  We use the generic character table of $G \cong \PSL(2,q)$
  to compute $T(G)$.
  If $q \equiv 3 \bmod 4$ then $T(G) = q (q+1)/2$,
  and we can choose $N = l(q-1)/2$.
  If $q \equiv 1 \bmod 4$ then $T(G) = 1 + q (q+1)/2$,
  and we can choose $N = l(q+1)/2$.
  In each case, $[G:H] \leq N < T(G)$ holds.
\item
  $G = A_2(q)$.
  Then $|G| = q^3 (q^2-1)(q^3-1) / \gcd(3, q-1)$.
  We choose $N = l (q^2+q+1)$, which is smaller than $|G|/N$.
\item
  $G = B_2(q)$.
  Then $|G| = q^4 (q^2-1)(q^4-1) / 2$.
  We choose $N = l (q^2+1)$, which is smaller than $|G|/N$.
\item
  $G = {}^2A_2(q)$.
  Then $|G| = q^3 (q^2-1)(q^3+1) / \gcd(3, q+1)$.
  We choose $N = l(q^2-q+1)$, which is smaller than $|G|/N$.
\item
  $G = G_2(q)$.
  Then $|G| = q^6 (q^6-1) (q^2-1)$.
  We choose $N = l (q^2+q+1)$, which is smaller than $|G|/N$.
\item
  $G = {}^3D_4(q)$.
  Then $|G| = q^{12} (q^8+q^4+1) (q^6-1) (q^2-1)$.
  We choose $N = l (q^4-q^2+1)$, which is smaller than $|G|/N$.
\item
  $G = {}^2G_2(q)$, $q = 3^{2m+1}$.
  Then $|G| = q^3 (q^3+1) (q-1)$.
  We choose $N = l (q^2-q+1)$, which is smaller than $|G|/N$.
\end{itemize}
\end{proof}

\begin{rem}
Proposition~\ref{prop:indexpr} can be proved without using
Proposition~\ref{prop:SylowReductive},
but then the arguments are more technical.
\end{rem}


\begin{rem}
For a fixed integer $b>1$, primes of the form $(b^n-1)/(b-1)$
where $n$ is a positive integer (necessarily a prime) are
called repunit primes to base $b$. For $b=2$, these are the 
well-known Mersenne primes. One expects that there are 
infinitely many repunit primes to base $b$, for each $b$ which
is not a proper power. Examples of these primes can be found
in \cite{Du}.
\end{rem}

The next result can be viewed as a kind of converse to
Proposition~\ref{prop:indexpr}.

\begin{prop}\label{prop:agl1}
Let $p$ be a prime, let $n$ be a positive integer, set $q = p^n$ and write
$q - 1 = r s$ with positive integers $r,s$. Moreover, let $G$ be the unique
subgroup of order
$q r$ in $\Gamma = \AGL(1, q)$, and let $H$ be a subgroup of order
$p^{n-1}$ in $G$. Then $H$ is rich in $G$ if and only if $s$ is a
divisor of $p - 1$ and $\gcd(r,s,n) = 1$.
\end{prop}

\begin{proof}
Let $C$ be a complement of $N := O_p(\Gamma)$ in $\Gamma$. Then $C$ is cyclic
of order $q-1$, and $R := G \cap C$ is a complement of $N$ in $G$. We claim 
that $H$ is rich in $G$ if and only if $R$ acts transitively on the set of
subgroups of order $p^{n-1}$ in $G$. For the proof, recall that $G$ has $r$
linear characters and $s$ irreducible characters of degree $r$. Every linear
character of $G$ is trivial on $N$ and thus on $H$. Every non-linear 
irreducible character of $G$ has the form $\chi = \lambda^G$ where $\lambda$
is a nontrivial irreducible character of $N$. Thus $\chi_N = \sum_{g \in R}
\lambda^g$, and the kernel of $\lambda$ is a subgroup of order $p^{n-1}$ in
$G$. Hence we have $[\chi_H,1_H] > 0$ if and only if $H$ is the kernel of 
$\lambda^g$ for some $g \in R$. From this the claim follows easily.

Obviously $C$ permutes the $(q-1)/(p-1)$ subgroups of order $p^{n-1}$ in $N$
transitively, and $D := N_C(H)$ has order $p-1$. Hence $R$ permutes the 
subgroups of order $p^{n-1}$ in $G$ transitively if and only if $C = RD$. We
claim that this is the case if and only if $s$ divides $p-1$ and $\gcd(r,s,n)
= 1$. For the proof, note that 
$$(q-1)/(p-1) = p^{n-1} + \ldots + p + 1 \equiv n \pmod{p-1}.$$
Now suppose that $C = RD$. Comparing orders on both sides we obtain
$$rs = q-1 = r(p-1)/\gcd(r,p-1).$$
Thus $s = (p-1)/\gcd(r,p-1)$ divides $p-1$. Assume that $\gcd(r,s,n) > 1$, 
and choose a common prime divisor $t$ of $r$, $s$ and $n$. Then the Sylow
$t$-subgroup $R_t$ of $R$ is properly contained in the Sylow $t$-subgroup
$C_t$ of $C$. Since $t$ divides $n$ and thus $(q-1)/(p-1)$, the Sylow
$t$-subgroup $D_t$ of $D$ is also properly contained in $C_t$. Since $C_t$ is cyclic this implies $R_tD_t < C_t$ which is a contradiction. 

Conversely, suppose that $s$ divides $p-1$ and $\gcd(r,s,n) = 1$. Then we 
have $(q-1)/(p-1) \equiv n \pmod{s}$ and thus $\gcd(r,s,(q-1)/(p-1)) = 1$. 
Let $t$ be a prime divisor of $q-1 = rs$. If $t$ divides $r$ but not $s$
then $R$ contains $C_t$. If $t$ divides $r$ and $s$ then $t$ does not 
divide $(q-1)/(p-1)$.
Thus $D$ contains $C_t$. Finally, if $t$ does not divide $r$ then $|C_t|$
divides $s$ and thus $p-1$. Hence $D$ contains $C_t$. We conclude that 
$RD = C$.
\end{proof}


\section{Rich subgroups in symmetric and alternating groups}%
\label{sect:symmetric}

In this section we consider rich subgroups of the symmetric and alternating 
groups. We start with the following result.

\begin{prop}\label{transitiverich}
 Let $H$ be a rich subgroup of a permutation group $G$.
 Then $H$ is intransitive.
\end{prop}

\begin{proof}
 We can assume that $G$ is a subgroup of the symmetric group $S_n$
 where $n$ is a positive integer.
 By~\cite[Proposition~2.9]{partI}, $H$ is rich in $S_n$.
 Since $H < S_n$ we have $n>1$.
 Let $\pi$ be the natural permutation character of $S_n$.
 Then $\pi = 1_{S_n} + \chi$ where $\chi$ is a nontrivial irreducible character
 of $S_n$. Assume that $H$ is transitive.
 Then
 \[
    1 = [\pi_H,1_H] = [1_H + \chi_H, 1_H] = 1 + [\chi_H,1_H].
 \]
 This gives the contradiction $[\chi_H,1_H] = 0$. 
\end{proof}

We obtain the following consequence.

\begin{prop}\label{MaxAlternating}
 Let $H$ be a maximal subgroup of an alternating group $A_n$.
 Then $H$ is not rich in $A_n$ and in the symmetric group $S_n$.
\end{prop}

\begin{proof}
 Let $H < A_n$ be rich in $S_n$.
 By Proposition~\ref{transitiverich}, $H$ must be intransitive. 
 We may therefore assume that $H \subseteq S_k \times
 S_{n-k}$ for some $k \in \{1,\ldots,n-1\}$. This implies that $H = A_n 
 \cap (S_k \times S_{n-k})$; in particular, we have $|H| = k! (n-k)!/2$ and 
 $[S_n:H] = 2 {n \choose k}$. Since $H$ is rich in $S_n$, we must have 
 $|S_n| < [S_n:H]^2$, by \cite[Proposition 4.1]{partI}. This implies that 
 $n! < 4 {n \choose k}^2 \leq 4 {n \choose \lfloor n/2 \rfloor}^2$.
 However, one easily proves by induction that
 $m! \geq 4 {m \choose \lfloor m/2 \rfloor}^2$ for $m > 6$.
 Thus we conclude that $n \leq 6$.
 But in this case we easily get a contradiction from the inequality
 $T(S_n) \leq [S_n:H] = 2 {n \choose k}$ (cf.~\cite[Proposition 4.1]{partI}).
\end{proof}

In general, maximal subgroups of simple groups can be rich,
see~\cite[Tables 3 and 4]{partI}.

We now describe the rich \emph{cyclic} subgroups
of symmetric and alternating groups.

\begin{thm}\label{cyclicinS_n}
Let $\sigma$ be a nonidentity element of the symmetric group $S_n$
that is contained in the alternating group $A_n$.
(Note that by \cite[Corollary~2.8~(iii)]{partI},
any rich subgroup of $S_n$ is contained in $A_n$.)

(i) $\langle \sigma \rangle$ is rich in $S_n$,
except if $\sigma$ has the cycle type $(3, 1)$ or $(5, 3)$ or $(n)$.

(ii) $\langle \sigma \rangle$ is rich in $A_n$
if and only if it is rich in $S_n$.
\end{thm}

\begin{proof}
We use the notation of~\cite[Theorems~1, 2]{Staroletov},
in particular the cycle type of $\sigma$ is given
by the partition $\mu$ of $n$.

Then $\langle \sigma \rangle$ is not rich in $S_n$
if and only if one of the exceptional
cases (i)--(xi) in~\cite[Theorem~1]{Staroletov}
yields an irreducible representation $\rho$ of $S_n$
such that $1$ is not an eigenvalue of $\rho(\sigma)$.
Analogously,
$\langle \sigma \rangle$ is not rich in $A_n$
if and only if one of the exceptional
cases (i)--(x) in~\cite[Theorem~2]{Staroletov}
yields an irreducible representation $\rho$ of $A_n$
such that $1$ is not an eigenvalue of $\rho(\sigma)$.

Assume that we are in the case $S_n$.
Only the exceptional cases (iii), (iv), (v), (vii), and (x)
in~\cite[Theorem~1]{Staroletov} exclude the eigenvalue $1$.
Disregarding those $\mu$ that belong to elements outside $A_n$,
we get what is claimed in part (i).

Assume that we are in the case $A_n$.
Only the exceptional cases (ii), (iii), (vi), (vii), (ix), and (x)
in~\cite[Theorem~2]{Staroletov} exclude the eigenvalue $1$.
This yields the same $\mu$ as for the case $S_n$.
\end{proof}

In the following result, we consider the alternating group $A_n$ as a subgroup
of the symmetric group $S_k$, for $k \geq n$, using the canonical embedding.
We use an abbreviated notation of partitions, the entry $d^k$ denotes
$k$ parts of length $d$.

\begin{thm}\label{richA_n}
Let $n > 2$. With the convention above, $A_n$ is rich in $S_{(n-1)^2 + 1}$ 
but not in $S_{(n-1)^2}$.
\end{thm}

\begin{proof}
Let $\lambda = (\lambda_1,\ldots, \lambda_m)$ be a partition of 
$(n - 1)^2 + 1$. Then $\lambda_1 \geq n$ or $m \geq n$. Therefore,
by consecutively removing removable nodes, we can arrive at the partition
$(n)$ or $(1^n)$ of $n$. Hence the restriction $(V_\lambda)_{S_n}$ of the 
Specht module $V_\lambda$ contains the trivial or the
alternating $S_n$-module. Therefore $(V_\lambda)_{A_n}$ contains the 
trivial $A_n$-module.

On the other hand, consider the partition $\lambda = ((n - 1)^{n-1})$ of
$(n-1)^2$. The restriction of the Specht module $V_\lambda$ to $S_n$ involves 
neither the trivial nor the alternating module. Hence $A_n$ is not rich in $S_{(n-1)^2}$.
\end{proof}

\begin{rem}
The result above becomes false if we consider different embeddings of 
$A_n$ into larger symmetric groups:
For example, $S_6$ has a transitive subgroup isomorphic to $A_5$ which
becomes a rich subgroup of $S_{11}$ but not of $S_{10}$.
Similarly, $S_{10}$ has a transitive subgroup isomorphic to $A_5$ which
becomes a rich subgroup of $S_{12}$ but not of $S_{11}$. 

Also, $S_{15}$ has a transitive subgroup $H$ isomorphic to
$A_5$, and $H$ is not rich in $S_{17}$. On the other hand, every subgroup
of $S_{17}$ which is isomorphic to $A_5$ and not conjugate to $H$ is rich
in $S_{17}$. Altogether, $S_{17}$ has $14$ conjugacy classes of subgroups
isomorphic to $A_5$.
\end{rem}

The following consequence of Theorem~\ref{richA_n} could be of interest.

\begin{cor}
Let $G$ be a group of order $n$. Then $S_{n^2 + 2n + 2}$ has a rich subgroup
isomorphic to $G$. In particular, every finite group can be embedded into
another finite group as a rich subgroup.
\end{cor}

\begin{proof}
The regular permutation representation embeds $G$ into $S_n$ and hence
into $A_{n+2}$. Since $n^2 + 2n + 2 = (n + 2 - 1)^2 + 1$,
the claim follows then by Theorem~\ref{richA_n} above
and by~\cite[Proposition~2.9]{partI}.
\end{proof}


A natural question is which \emph{Young subgroups} of alternating groups
are rich in symmetric groups.
For a partition $\lambda = (\lambda_1, \lambda_2, \ldots, \lambda_k)$ of $n$,
we denote by $S_{\lambda}$ the Young subgroup
$S_{\lambda_1} \times S_{\lambda_2} \times \cdots \times S_{\lambda_k}$
of $S_n$,
and set $A_{\lambda} = S_{\lambda} \cap A_n$.
By Theorem~\ref{richA_n}, any nontrivial $A_{\lambda}$ is rich
in symmetric groups of sufficiently large degree,
and we can ask for the minimal degree of such a
rich embedding of $A_{\lambda}$.

It turns out that the answer can be formulated in terms of the
dominance order of partitions.
For a partition $\lambda = (\lambda_1, \lambda_2, \ldots)$
and a positive integer $i$, let $s_i(\lambda) = \sum_{k=1}^i \lambda_k$.
We say that $\lambda$ dominates a partition $\mu = (\mu_1, \mu_2, \ldots)$
if $s_i(\lambda) \geq s_i(\mu)$ holds for all $i$,
and denote this by $\lambda \dominates \mu$.
Dominance defines a lattice structure on the partitions of $n$,
see \cite[1.4.16]{JamesKerber}.

The key argument for deciding the richness of a Young subgroup of $S_n$
is as follows.

\begin{lem}\label{lemma:richyoung}
For a partition $\lambda$ of $n$,
$A_{\lambda}$ is rich in $S_n$ if and only if each partition $\mu$ of $n$
satisfies $\mu \dominates \lambda$ or $\mu' \dominates \lambda$.
\end{lem}

\begin{proof}
Let $H = A_{\lambda}$, $S = S_{\lambda}$, and $G = S_n$.
The character $(1_S)^G$ is afforded by the permutation module $M^{\lambda}$
of $G$, and its irreducible constituents correspond to the partitions
$\mu$ that dominate $\lambda$,
by \cite[Cor.~2.2.22]{JamesKerber},
a strong version of Young's rule.
We have $(1_H)^S = 1_S + \varepsilon_S$ for a linear character $\varepsilon_S$
of $S$ that extends to the sign character $\varepsilon_G$ of $G$,
thus $(1_H)^G = (1_S)^G + \varepsilon_G \cdot (1_S)^G$.
The irreducible constituents of $\varepsilon_G \cdot (1_S)^G$ then
correspond to the conjugate partitions $\mu'$ such that $\mu$
occurs in $(1_S)^G$.
\end{proof}

\begin{example}\label{example:young}
Consider the partition $\lambda = (2^n)$ of $2n$, with $n > 1$.

For a partition $\mu$ of $2n$,
$\mu \dominates \lambda$ if and only if
$\lambda' = (n, n) \dominates \mu'$,
which happens if and only if $\mu$ has at most $n$ rows.
Analogously, $\mu' \dominates \lambda$ if and only if
$\mu$ has at most $n$ columns.
Since each partition of $2n$ has at most $n$ rows or at most $n$ columns,
$A_{\lambda}$ is rich in $S_{2n}$.
\end{example}

In the following, we generalize the idea of Example~\ref{example:young}.

\begin{lem}
Write $n = k^2 + 2x - e$ with $e \in \{ 0, 1 \}$ and $0 \leq 2x-e \leq 2k$.
(Then $e \leq x \leq k$.)
Set $l = \lfloor n/2 + 1 \rfloor$.

Consider the partitions
$\sigma = ((k+1)^x, k^{k-x}, x-e)$ and $\kappa = (l, 1^{n-l})$ of $n$,
and define $\Lambda = \Lambda_n = \inf(\sigma, \kappa)$.
Then $\Lambda$ and its conjugate $\Lambda' = \Lambda'_n$
have the following forms.

\begin{itemize}
\item[(I)]
  For odd $k$ and $0 \leq x \leq (k-1)/2$:
  \begin{eqnarray*}
      \Lambda  & = & ((k+1)^x, k^{(k+1)/2 - x}, 1^{k(k-1)/2 + x - e}), \\
      \Lambda' & = & ((k^2+1)/2+x-e, ((k+1)/2)^{k-1}, x).
  \end{eqnarray*}
\item[(II)]
  For odd $k$ and $(k+1)/2 \leq x \leq k$:
  \begin{eqnarray*}
      \Lambda  & = & ((k+1)^{(k+1)/2}, x-(k-1)/2, 1^{k(k-1)/2-1+x-e}), \\
      \Lambda' & = & ((k^2+1)/2+x-e, ((k+3)/2)^{x-(k+1)/2}, ((k+1)/2)^{(3k+1)/2-x}).
  \end{eqnarray*}
\item[(III)]
  For even $k$ and $e \leq x < k/2$:
  \begin{eqnarray*}
      \Lambda  & = & ((k+1)^x, k^{k/2-x}, k/2+1-e, 1^{k(k-1)/2+x-1}), \\
      \Lambda' & = & (k^2/2+x, (k/2+1)^{k/2-e}, (k/2)^{k/2-1+e}, x).
  \end{eqnarray*}
\item[(IV)]
  For even $k$ and $k/2 \leq x \leq k$:
  \begin{eqnarray*}
      \Lambda  & = & ((k+1)^{k/2}, x-e+1, 1^{k(k-1)/2+x-1}), \\
      \Lambda' & = & (k^2/2 + x, (k/2+1)^{x-e}, (k/2)^{k-x+e}).
  \end{eqnarray*}
\end{itemize}

(Note that in (I) and (III), the border case $x = 0$ means
that $\Lambda$ has no rows of length $k+1$;
in (II), the border case $x = (k+1)/2$ means
that all rows of $\Lambda$ have length $k+1$ or $1$;
and in (IV), the border case $(x, e) = (k, 0)$ means
that $\Lambda$ has $k/2+1$ rows of length $k+1$.)
\end{lem}

\begin{proof}
We have $s_i(\sigma) = n$ for $i \geq k+1$,
$s_i(\kappa) = n$ for $i \geq n-l+1 = \lfloor (n+1)/2 \rfloor$, and
the equalities
\begin{eqnarray*}
   s_i(\sigma) & = & \left\{ \begin{array}{lcl}
                          i(k+1) & ; & 1 \leq i \leq x \\
                          ik+x   & ; & x < i \leq k
                        \end{array} \right. \\
   s_i(\kappa) & = & \left\{ \begin{array}{lcl}
           (k^2-1)/2+x+i & ; & k \textrm{\ odd}, 1 \leq i \leq (n+1)/2 \\
           k^2/2+x+i-e   & ; & k \textrm{\ even}, 1 \leq i \leq (n+1)/2
                        \end{array} \right.
\end{eqnarray*}
hold.

\begin{itemize}
\item
  If $i \leq k/2$ then
  \[
     s_i(\sigma) \leq ik + x \leq k^2/2 + x \leq s_1(\kappa) \leq s_i(\kappa),
  \]
  and if $i = (k+1)/2$ then $s_i(\sigma) \leq k(k+1)/2 + x = s_i(\kappa)$.

  Hence $s_i(\sigma) \leq s_i(\kappa)$ holds for
  $1 \leq i \leq \lfloor (k+1)/2 \rfloor$,
  which implies $(\Lambda)_i = \sigma_i$ for these $i$.
\item
  Consider $i = \lfloor (k+1)/2 + 1 \rfloor$.
  \begin{itemize}
  \item
    If $i = k/2 + 1$ then $s_i(\kappa) = k(k+1)/2 + x + 1 - e$.

    For $i \leq x$,
    we have
    \[
      s_i(\sigma) = i(k+1) = k(k+3)/2 + 1 \geq s_i(\kappa)
    \]
    since $x-e \leq k$.

    For $x < i$,
    we have
    \[
      s_i(\sigma) = ik + x = k^2/2 + k + x \geq s_i(\kappa).
    \]
  \item
    If $i = (k+3)/2$ then
    $s_i(\kappa) = (k^2-1)/2 + x + (k+3)/2 = k(k+1)/2 + x + 1$.

    For $i \leq x$,
    we have
    \[
      s_i(\sigma) = \frac{(k+3)(k+1)}{2}
                  > \frac{k(k+3)}{2} + 1
                  \geq \frac{k^2-1}{2} + x + i = s_i(\kappa),
    \]
    since $x \leq k$.

    For $x < i$,
    we have
    \[
      s_i(\sigma) = k(k+3)/2 + x \geq s_i(\kappa).
    \]
  \end{itemize}
  Thus we get $s_i(\kappa) \leq s_i(\sigma)$
  and hence $(\Lambda)_i = s_i(\kappa) - s_{i-1}(\sigma)$,
  which is equal to
  \[
      \begin{array}{llcl}
        ((k^2-1)/2+x+i) - ((i-1)k+x) & = & 1
            & \textrm{in (I),} \\
        ((k^2-1)/2+x+i) - (i-1)(k+1) & = & x - (k-1)/2
            & \textrm{in (II),} \\
        (k^2/2+x+i-e) - ((i-1)k + x) & = & k/2 + 1 - e
            & \textrm{in (III),} \\
        (k^2/2+x+i-e) - (i-1)(k+1)   & = & x+1-e
            & \textrm{in (IV).}
      \end{array}
  \]
\item
  For $\lfloor (k+1)/2 + 1 \rfloor < i$,
  we have $\kappa_i \leq 1$, which implies $s_i(\kappa) \leq s_i(\sigma)$
  and thus $(\Lambda)_i \leq \kappa_i \leq 1$.
\end{itemize}
\end{proof}

\begin{example}
The values of $\Lambda_n$ for small $n$ are listed in Table~\ref{Table:Lambda}.

\begin{table}           
\caption{Partitions $\Lambda_n$, for small $n$}
\label{Table:Lambda}
\hspace*{2cm}
\begin{minipage}{3cm}  
\scriptsize           
\[
  \begin{array}[t]{r|l}
   n & \Lambda_n \\ \hline 
   1 & (1^1) \\
   2 & (2^1) \\
   3 & (2^1,1^1) \\
   4 & (2^2) \\
   5 & (3^1,1^2) \\
   6 & (3^1,2^1,1^1) \\
   7 & (3^1,2^1,1^2) \\
   8 & (3^2,1^2) \\
   9 & (3^2,1^3) \\
  10 & (4^1,3^1,1^3)
  \end{array}
\]
\end{minipage}
\
\begin{minipage}{3cm}
\scriptsize
\[
  \begin{array}[t]{r|l}
   n & \Lambda_n \\ \hline 
  11 & (4^1,3^1,1^4) \\
  12 & (4^2,1^4) \\
  13 & (4^2,1^5) \\
  14 & (4^2,2^1,1^4) \\
  15 & (4^2,2^1,1^5) \\
  16 & (4^2,3^1,1^5) \\
  17 & (5^1,4^1,2^1,1^6) \\
  18 & (5^1,4^1,3^1,1^6) \\
  19 & (5^2,2^1,1^7) \\
  20 & (5^2,3^1,1^7)
  \end{array}
\]
\end{minipage}
\
\begin{minipage}{3cm}
\scriptsize
\[
  \begin{array}[t]{r|l}
   n & \Lambda_n \\ \hline 
  21 & (5^2,3^1,1^8) \\
  22 & (5^2,4^1,1^8) \\
  23 & (5^2,4^1,1^9) \\
  24 & (5^3,1^9) \\
  25 & (5^3,1^{10}) \\
  26 & (6^1,5^2,1^{10}) \\
  27 & (6^1,5^2,1^{11}) \\
  28 & (6^2,5^1,1^{11}) \\
  29 & (6^2,5^1,1^{12}) \\
  30 & (6^3,1^{12})
  \end{array}
\]
\end{minipage}
\
\begin{minipage}{3cm}
\scriptsize
\[
  \begin{array}[t]{r|l}
   n & \Lambda_n \\ \hline 
  31 & (6^3,1^{13}) \\
  32 & (6^3,2^1,1^{12}) \\
  33 & (6^3,2^1,1^{13}) \\
  34 & (6^3,3^1,1^{13}) \\
  35 & (6^3,3^1,1^{14}) \\
  36 & (6^3,4^1,1^{14}) \\
  37 & (7^1,6^2,3^1,1^{15}) \\
  38 & (7^1,6^2,4^1,1^{15}) \\
  39 & (7^2,6^1,3^1,1^{16}) \\
  40 & (7^2,6^1,4^1,1^{16})
  \end{array}
\]
\end{minipage}
\end{table}
\end{example}

\begin{lem}\label{lem:domination}
For any partition $\mu$ of $n$,
$\mu \dominates \Lambda_n$ or $\mu' \dominates \Lambda_n$.
\end{lem}

\begin{proof}
In general, $\Lambda = ((k+1)^p, k^q, r, 1^s)$ for nonnegative integers
$p, q, r, s$, with $p \leq x$, $p+q \in \{ k/2, (k+1)/2, k/2+1 \}$,
and $r, s \geq 1$.

\textbf{Step 1:}
We claim that $\mu \dominates \Lambda$
if $s_p(\mu) \geq s_p(\Lambda)$ and
$s_{p+q}(\mu) \geq s_{p+q}(\Lambda)$ and
$s_{p+q+1}(\mu) \geq s_{p+q+1}(\Lambda)$.

For otherwise, we have $s_i(\mu) < s_i(\Lambda)$ for some $i$.
We distinguish three cases:
\begin{itemize}
\item
  If $1 \leq i \leq p$ then $\mu_i < \Lambda_i$ and then
  \[
     s_p(\mu) = s_i(\mu) + \sum_{j=i+1}^p \mu_j
              < s_i(\Lambda) + (p-i)\Lambda_1
              = s_p(\Lambda).
  \]
\item
  If $p < i \leq p+q$ then $\mu_i < \Lambda_i$ and then
  \[
     s_{p+q}(\mu) = s_i(\mu) + \sum_{j=i+1}^{p+q} \mu_j
                  < s_i(\Lambda) + (p+q-i)\Lambda_{p+1}
                  = s_{p+q}(\Lambda).
  \]
\item
  If $i > p+q$ then $(\Lambda)_i = 1$, and either $\mu_i \geq 1$
  or $s_i(\mu) = n \geq s_i(\Lambda)$.
\end{itemize}

\textbf{Step 2:}
We show that $\Lambda' \dominates \mu$ in all other cases,
which is equivalent to $\mu' \dominates \Lambda$.

\begin{itemize}
\item
  Assume $s_p(\mu) < s_p(\Lambda)$.

  \begin{itemize}
  \item
    If $1 \leq i \leq p$ then
    \[
       s_i(\mu) \leq s_p(\mu) - (p-i)
                \leq s_p(\Lambda) - 1 - p + i
                =    pk + i - 1
    \]
    holds.
    Since $s_i(\Lambda') \geq (\Lambda')_1 + (i-1)$,
    it suffices to show that $pk \leq (\Lambda')_1$.

    In all situations except (IV) with $(x, e) = (k, 0)$,
    we have $p \leq \min\{ x, (k+1)/2 \}$ and
    $(\Lambda')_1 \geq x + (k^2-1)/2$,
    thus
    \[
       pk = p  + (k-1)p \leq x + (k-1)(k+1)/2 \leq (\Lambda')_1.
    \]
    In the exceptional situation, we have $pk = (k/2+1)k = (\Lambda')_1$.
  \item
    If $p+1 \leq i \leq k$ then $\mu_i \leq k$ and
    \[
       s_i(\mu) \leq s_p(\mu) + (i-p)k \leq p(k+1) - 1 + (i-p)k = ik + p - 1.
    \]
    Since $s_i(\Lambda') \geq (\Lambda')_1 + (i-1)(k+1)/2$
    and $(\Lambda')_1 \geq (k^2-1)/2 + x$, we have
    \begin{eqnarray*}
       s_i(\Lambda') - s_i(\mu) & \geq & (k^2-1)/2 + x + (i-1)(k+1)/2 - ik - p + 1 \\
                                & = & k^2/2 + x - ik/2 - k/2 + i/2 - p \\
                                & = & (k-1)(k-i)/2 + (x - p) \\
                                & \geq & 0.
    \end{eqnarray*}
  \end{itemize}
\item
  Assume $s_p(\mu) \geq s_p(\Lambda)$ and
  $s_{p+q}(\mu) < s_{p+q}(\Lambda)$.

  Then we are in one of the situations (I), (III), or (IV)
  (where $x = k + e -1$ and $q = 1$ in situation (IV)).
  Thus $p+q \in \{ k/2, (k+1)/2 \}$ and $x \geq p$.

  \begin{itemize}
  \item
    If $1 \leq i \leq p+q$ then
    \begin{eqnarray*}
       s_i(\mu) & \leq & s_{p+q}(\mu) - (p+q-i) \\
                & \leq & s_{p+q}(\Lambda) - 1 - (p+q) + i \\
                & = & p(k+1) + qk - (p+q) + i-1 \\
                & \leq & (p+q)(k-1) + x + (i-1) \\
                & \leq & (k^2-1)/2 + x + (i-1) \\
                & = & (\Lambda')_1 + (i-1) \\
                & \leq & s_i(\Lambda').
    \end{eqnarray*}
  \item
    If $p+q < i \leq k$ then $\mu_i \leq k-1$.
    (Note that otherwise
    $s_{p+q}(\mu) \geq s_p(\mu) + qk \geq p(k+1) + qk = s_{p+q}(\Lambda)$
    would hold, contrary to our assumption.)

    In the situations (I) and (III),
    we have
    $s_{p+q}(\Lambda) = k(p+q) + x$ and thus

    \begin{eqnarray*}
       s_i(\mu) & \leq & s_{p+q}(\mu) + (i-(p+q))(k-1) \\
                & \leq & s_{p+q}(\Lambda) -1 + (i-(p+q))(k-1) \\
                & = & k(p+q) + x - 1 + (i-(p+q))(k-1) \\
                & = & x - 1 + i(k-1) + p+q
    \end{eqnarray*}

    and hence $p+q \in \{ (k+1)/2, k/2+1 \}$ implies

    \begin{eqnarray*}
       s_i(\Lambda') - s_i(\mu) & \geq & (k^2-1)/2 + x + (i-1)(p+q)
                                     - (x - 1 + i(k-1) + p+q) \\
                                & = & (k^2+1)/2 - i(k-1) + (i-2)(p+q) \\
                                & \geq & (k^2+1)/2 - i(k-1) + (i-2)(k+1)/2 \\
                                & = & ((k-i)(k-2) + i - 1)/2 \\
                                & \geq & 0.
    \end{eqnarray*}

    In situation (IV),
    we have $p+q = k/2+1$, $x = k + e -1$, and
    $s_{p+q}(\Lambda) = (k+1)k/2 + k$ and thus

    \begin{eqnarray*}
       s_i(\mu) & \leq & s_{p+q}(\mu) + (i-(p+q))(k-1) \\
                & \leq & s_{p+q}(\Lambda) -1 + (i-(p+q))(k-1) \\
                & = & (k+1)k/2 + k - 1 + (i-(k/2+1))(k-1) \\
                & = & i(k-1)+k,
    \end{eqnarray*}

    and $i > k/2 + 1 \geq 2$ implies

    \begin{eqnarray*}
       s_i(\Lambda') - s_i(\mu) & \geq & k^2/2 + k-1 + (i-1)k/2 -(i(k-1)+k) \\
                                & = & k/2(k - i - 1) +i-1 \\
                                & \geq & 0.
    \end{eqnarray*}
  \end{itemize}
\item
  Assume $s_p(\mu) \geq s_p(\Lambda)$ and
  $s_{p+q}(\mu) \geq s_{p+q}(\Lambda)$ and
  $s_{p+q+1}(\mu) < s_{p+q+1}(\Lambda)$.

  Then we are in one of the situations (II), (III), or (IV),
  where $x < k+e-1$ holds in situation (IV).
  Thus $p+q \in \{ k/2, (k+1)/2 \}$.

  In situation (II), $s_{p+q+1}(\Lambda) = k(k+1)/2 + x + 1$
  and $\Lambda_{p+q+1} = x-(k-1)/2 \leq (k+1)/2$.
  In situation (III) or (IV), $s_{p+q+1}(\Lambda) = k(k+1)/2 + x - e + 1$.
  In situation (III), $\Lambda_{p+q+1} = k/2+1-e$.
  In situation (IV), $\Lambda_{p+q+1} = x-e+1 \leq k-1$.

  \begin{itemize}
  \item
    If $1 \leq i \leq p+q+1$ then $p+q \geq k/2$ implies
    \begin{eqnarray*}
       s_i(\mu) & \leq & s_{p+q+1}(\mu) - (p+q+1-i) \\
                & \leq & s_{p+q+1}(\Lambda) - 1 - (p+q+1-i) \\
                & \leq & k(k+1)/2 + x - (p+q+1-i) \\
                & \leq & k(k+1)/2 + x + i -1 -k/2 \\
                & = & (k^2-1)/2 + x + (i-1) + 1/2 \\
                & \leq & s_i(\Lambda') + 1/2,
    \end{eqnarray*}
    thus $s_i(\mu) \leq s_i(\Lambda')$.
  \item
    Let $k/2 + 1 = p+q+1 < i \leq k$.

    We have
    \[
      \mu_{p+q+1} = s_{p+q+1}(\mu) - s_{p+q}(\mu)
                  < s_{p+q+1}(\Lambda) - s_{p+q}(\Lambda)
                  = (\Lambda)_{p+q+1}.
    \]
    In situations (II) and (III),
    $\mu_i \leq \mu_{p+q+1} \leq \Lambda_{p+q+1} - 1 \leq (\Lambda')_i$.
    Hence $s_{p+q+1}(\mu) \leq s_{p+q+1}(\Lambda')$ implies
    $s_i(\mu) \leq s_i(\Lambda')$, by induction.

    Thus it is left to show that $s_i(\mu) \leq s_i(\Lambda')$ holds
    also in situation (IV).
    We have $p+q = k/2$ and $\mu_{p+q+1} \leq x-e \leq k-2$,
    in particular $k > 2$.
    \begin{eqnarray*}
      s_i(\Lambda') & \geq & k^2/2 + x + (i-1)k/2 \\
      s_i(\mu)      & \leq & s_{p+q+1}(\mu) + (i-(p+q+1))(k-2) \\
                    & \leq & s_{p+q+1}(\Lambda) + (i-(p+q+1))(k-2) \\
                    & \leq & k(k+1)/2 + x +1 + (i-(p+q+1))(k-2) \\
    \end{eqnarray*}
    and thus
    \begin{eqnarray*}
      s_i(\Lambda') - s_i(\mu) & \geq & k^2/2 + x + (i-1)k/2 \\
                               &      & - (k(k+1)/2 + x +1 + (i-(p+q+1))(k-2)) \\
                               & = & k(k-2)/2 - i(k-4)/2 - 3.
    \end{eqnarray*}
    Since $k \geq 4$,
    the right hand side is non-increasing for increasing $i$.
    For $i = k+1$ we get
    \[
       k(k-2)/2 - (k+1)(k-4)/2 - 3 = k/2 - 1 \geq 0,
    \]
    and we are done. \qedhere
  \end{itemize}
\end{itemize}
\end{proof}

\begin{thm}\label{RichYoung}
For a partition $\lambda$ of $n \geq 3$,
$A_{\lambda}$ is rich in $A_n$ (and hence in $S_n$)
if and only if $\Lambda_n \dominates \lambda$.
\end{thm}

\begin{proof}
By Lemma~\ref{lemma:richyoung} and Lemma~\ref{lem:domination},
$A_{\Lambda_n}$ is rich in $S_n$,
as well as $A_{\lambda}$ for each $\lambda$ with the property
$\Lambda_n \dominates \lambda$.

To show the converse, observe that
we have $\sigma \dominates \sigma'$ and $\kappa \dominates \kappa'$.
Thus $\sigma$ and $\kappa$ dominate each $\lambda$ for which
$A_{\lambda}$ is rich in $S_n$.
As a consequence, also $\Lambda_n = \inf(\sigma, \kappa)$
dominates each $\lambda$ for which $A_{\lambda}$ is rich in $G = S_n$.

It remains to show that $H = A_{\lambda}$ is rich in $A = A_n$
whenever $\Lambda_n \dominates \lambda$.

For that, set $\pi = (1_H)^A$ and consider $\phi \in \Irr(A)$.
If $\phi = \chi_A$ for some $\chi \in \Irr(G)$, where $G = S_n$, then
we know that $[\pi, \phi] = [\pi^G, \chi] > 0$.
Otherwise, $\phi + \phi' = \chi_A$ for some $\chi \in \Irr(G)$,
where $\phi, \phi' \in \Irr(A)$ are conjugate under the action of $G$,
and then $\phi$ and $\phi'$ differ at most on those conjugacy classes
of $A$ that fuse in $G$ with another class of $A$.
These classes belong to cycle types with pairwise different (and odd)
cycle lengths.

The largest possible cycle length of an element in $A_{\lambda}$ is $k+1$,
where $k = \lfloor \sqrt{n} \rfloor$.
If $k = 1$ then $n \leq 3$ holds, and nothing is to show because
$A_{\lambda}$ is trivial.
For $k \geq 2$, summing up the odd integers from $1$ to $k$
yields
\[
   1 + 3 + \cdots + (2\lfloor (k+1)/2 \rfloor - 1)
   = \lfloor (k+1)/2 \rfloor^2 < k^2 \leq n.
\]
Thus $A_{\lambda}$ does not contain elements from the classes of $A$
that fuse in $G$ with another class of $A$,
hence $\pi$ is zero on each of these classes, which implies
\[
   [\pi, \phi] = [\pi, \phi']
               = \frac{1}{2} [\pi, \chi_A]
               = \frac{1}{2} [\pi^G, \chi] > 0,
\]
and we are done.
\end{proof}

\section{The total degree}\label{sect:degrees}

We recall that $T(G) = \sum_{\chi \in \Irr(G)} \chi(1)$ is called the total
degree of a finite group $G$. Also, we always set $b(G) := \max\{\chi(1):
\chi \in \Irr(G)\}$. These notions can be generalized by replacing ordinary
representations by projective representations. For an element $\alpha$ in
the Schur multiplier of $G$, we denote by $\Irr(G,\alpha)$ the corresponding 
set of irreducible projective characters, and we set $b(G,\alpha) := 
\max\{\chi(1): \chi \in \Irr(G,\alpha)\}$ and $T(G,\alpha) := \sum_{\chi \in 
\Irr(G,\alpha)} \chi(1)$. These notions are useful when investigating $T(G)$
in the presence of a normal subgroup of $G$. 

\begin{prop}\label{TGClifford} 
Let $N$ be a normal subgroup of a finite group $G$.
For $\phi \in \Irr(N)$, 
let $G_\phi$ denote the stabilizer of $\phi$ in $G$ and 
let $\alpha_\phi$ denote the element in the Schur multiplier
of $G_\phi/N$ determined by $\phi$.
Then we have 
$$T(G) = \sum_{\phi \in {\Irr}(N)} T(G_\phi|\phi).$$
Moreover, for $\phi \in {\Irr}(N)$, we have
$$T(G_\phi|\phi) = T(G_\phi/N,\alpha_\phi) \phi(1).$$
\end{prop}

\begin{proof} 
 Let $R$ be a set of representatives for the $G$-orbits
on ${\rm Irr}(N)$. Then Clifford theory implies:
\[
   T(G) = \sum_{\phi \in R} T(G|\phi) 
        = \sum_{\phi \in R} [G:G_\phi] \cdot T(G_\phi|\phi)
        = \sum_{\phi \in {\rm Irr}(N)} T(G_\phi|\phi).
\]
Each character $\phi \in {\rm Irr}(N)$ extends to a projective character 
$\hat\phi$ of $G_\phi$. Let $\alpha_\phi$ be the corresponding element in 
the Schur multiplier of $G_\phi/N$.
Then the degrees of the characters in ${\rm Irr}(G_\phi|\phi)$ are 
the products of $\phi(1)$ with the degrees of the characters in 
$\Irr(G_\phi/N,\alpha_\phi)$.
\end{proof}

\begin{rem}\label{TGinequality}
Suppose we are in the situation of Proposition~\ref{TGClifford}.

\begin{itemize}
\item[(i)]
  For each $\phi \in {\rm Irr}(N)$:
  \begin{eqnarray*}
    |G_\phi/N| & = &    \sum\nolimits_{\chi\in\Irr(G_\phi/N,\alpha_\phi)}
                           \chi(1)^2 \\
               & \leq & \sum\nolimits_{\chi \in {\Irr}(G_\phi/N,\alpha_\phi)} 
                           \chi(1) \cdot b(G_\phi/N,\alpha_\phi) \\
               & = &    T(G_\phi/N,\alpha_\phi) \cdot b(G_\phi/N,\alpha_\phi),
  \end{eqnarray*}
  where equality holds if and only if all characters in 
  $\Irr(G_\phi/N,\alpha_\phi)$ have the same degree. (If $\alpha_\phi$ is 
  trivial then this means that $G_\phi/N$ is abelian.)  
  Moreover, we have
  \[
     b(G_\phi/N,\alpha_\phi) \leq T(G_\phi/N, \alpha_\phi),
  \]
  with equality if and only if $|\Irr(G_\phi/N,\alpha_\phi)| = 1$.
  (If $\alpha_\phi$ is trivial then this means that $G_\phi = N$.)

\item[(ii)]
  Similarly, we have
  \[
     b(G_\phi/N,\alpha_\phi)^2 \leq |G_\phi/N|,
  \]
  with equality if and only if $|\Irr(G_\phi/N, \alpha_\phi)| = 1$. 
  We conclude that
  \begin{eqnarray*}
    T(G) - T(G/N) & = &    \sum\nolimits_{1_N \neq \phi \in \Irr(N)} 
                              T(G_\phi/N,\alpha_\phi) \cdot \phi(1) \\
                  & \geq & \sum\nolimits_{1_N \neq \phi \in \Irr(N)} 
                              |G_\phi/N| \cdot \phi(1) / b(G_\phi/N,\alpha_\phi) \\
                  & \geq & \sum\nolimits_{1_N \neq \phi \in \Irr(N)} 
                              |G_\phi/N|^{1/2} \cdot \phi(1) \cr
                  & \geq & T(N) - 1,
  \end{eqnarray*}
  where all inequalities are equalities if and only if
  $G_\phi = N$ for every nontrivial $\phi \in \Irr(N)$.

\item[(iii)]
  Suppose that $1 < N < G$. Then $T(G) = T(G/N) + T(N) - 1$ holds
  if and only if $G$ is a Frobenius group with Frobenius kernel $N$.

  Note that $G$ is a Frobenius group with kernel $N$ (where $1 < N < G$)
  if and only if $C_G(x) \subseteq N$ holds for all $1 \not= x \in N$,
  see~\cite[pp.~99-100]{IsaacsCT}.

  An equivalent character theoretic condition is that $\phi^G \in \Irr(G)$
  holds for all nontrivial $\phi \in \Irr(N)$.
  One direction holds by~\cite[Thm.~(6.34)]{IsaacsCT},
  the other direction follows from the same argument.
  Namely, if $\phi^G \in \Irr(G)$ holds for all nontrivial $\phi \in \Irr(N)$
  then $G$ acts semiregularly on $\Irr(N) \setminus \{ 1_N \}$ and hence
  on the nontrivial conjugacy classes of $N$;
  this implies $C_G(x) \subseteq N$ for all $x \in N \setminus \{ 1 \}$.

  Now apply part~(ii).
\end{itemize}
\end{rem}

Specializing to the case $N = 1$ we obtain the following.

\begin{cor}\label{TGinequalities}
 For any finite group $G$, we have
 $$b(G)^2 \leq |G| \leq T(G) \cdot b(G) \leq T(G)^2.$$
\end{cor}

\begin{rem}\label{Berko}
Since $|G|$ is divisible by $d := b(G)$, one can write $|G| = d(d+e)$ where
$e$ is a nonnegative integer. Obviously, $e=0$ holds if and only if $G = 1$.
Berkovich \cite{B} proved that $e=1$ if and only if $G$ is a $2$-transitive
Frobenius group or $|G| = 2$. Hung, Lewis and Schaeffer Fry \cite{HLSF}
showed that $|G| \leq e^4 - e^3$ for $e \geq 2$, and Snyder \cite{S} 
classified the finite groups with $e=2$ or $e=3$. These results can be
used in order to sharpen the first inequality in 
Corollary~\ref{TGinequalities}.

 We also note that 
 $|G| = T(G)b(G)$ holds if and only if all irreducible characters of $G$
 have the same degree, i.~e., if and only if $G$ is abelian. 
 
 Moreover, $T(G) = b(G)$ holds if and only if $G = 1$. Similarly, $T(G) = 
 b(G)+1$ holds if and only if $|G| = 2$, and $T(G) = b(G) + 2$ holds
 if and only if $G$ is isomorphic to $A_3$ or $S_3$.
 Stronger inequalities follow from \cite[Proposition~4.1]{partI}.
\end{rem}

Next we specialize Proposition~\ref{TGClifford} to the case of a central
subgroup.

\begin{prop}\label{TGcentral}
 Let $N$ be a central subgroup of a finite group $G$. Then we have
 \begin{eqnarray*}
  T(G) - T(G/N) & \geq & (|N| - 1) \cdot |G/N| / b(G) \\
                & \geq & (|N| - 1) \cdot |G/N|^{1/2}.
 \end{eqnarray*}
 Moreover, $T(G) - T(G/N) = (|N| - 1) \cdot |G/N| / b(G)$ holds if and only 
 if $\chi(1) = b(G)$ for every $\chi \in \Irr(G) \setminus \Irr(G/N)$;
 here we identify characters of $G/N$ with their inflations to $G$.
\end{prop}

\begin{proof}
 Proposition~\ref{TGClifford} implies that
 $$T(G) - T(G/N) = \sum_{1_N \neq \phi \in \Irr(N)} T(G|\phi),$$
 and Remark~\ref{TGinequality} shows that, for $\phi \in \Irr(N)$, we
 have
 $$T(G|\phi) \cdot b(G) = \sum_{\chi \in \Irr(G|\phi)} \chi(1) \cdot b(G) \geq 
 \sum_{\chi \in \Irr(G|\phi)} \chi(1)^2 = |G/N|.$$
This proves the first inequality. The second inequality follows since 
$b(G)$ is a divisor of $|G/N|$ with $b(G)^2 \leq |G/N|$. The 
statement about equality is clear.
\end{proof}

\begin{rem}
If the normal subgroup $N$ is \emph{not} central in $G$
then the bounds from Remark~\ref{TGcentral} do not hold in general,
but we have the weaker bound
$T(G) - T(G/N) \geq \sqrt{|N|-1} \sqrt{|G/N|}$,
which follows from
\[
   T(G) - T(G/N) = \sum_\chi \chi(1)
                 = \sqrt{(\sum_\chi \chi(1))^2}
                 \geq \sqrt{\sum_\chi \chi(1)^2}
                 = \sqrt{|G| - |G/N|},
\]
where the summations run over the set
$\{ \chi \in \Irr(G) : N \not \subseteq \ker(\chi) \}$.
\end{rem}

\begin{rem}\label{TGpractice}
Again, assume that $N$ is central in $G$.

(i) In cases where $b(G)$ is not known precisely one may work with the 
following slightly weaker inequality:
$$T(G) - T(G/N) \geq (|N| - 1) \cdot c(|G/N|);$$
here $c(n)$ denotes the smallest positive divisor $d$ of a positive
integer $n$ with $d^2 \geq n$. 

(ii) Since $|G/N| \leq T(G/N)^2$ by Corollary~\ref{TGinequalities} we
also obtain $T(G) \geq |N| \cdot |G/N|^{1/2}$, in particular
$$|G| \leq T(G)^2/|Z(G)|.$$
\end{rem}

%

Now we look at another special case of Proposition~\ref{TGClifford}.

\begin{rem}\label{TGextend}
In the situation of Proposition~\ref{TGClifford}, suppose that $G/N$ is
abelian and that each $\phi \in \Irr(N)$ extends to its stabilizer 
$G_\phi$. Then $|G/N|$ divides $T(G)$.

In fact, we have
$$T(G) = \sum_{\phi \in \Irr(N)} T(G_\phi/N) \cdot \phi(1)
       = \sum_{\phi \in \Irr(N)} |G_\phi/N| \cdot \phi(1)
       = \sum_{\phi \in R} |G/N| \cdot \phi(1)$$
where $R$ is a set of representatives for the $G$-orbits on ${\rm Irr}(N)$. 
The claim follows.
\end{rem}

Note that the extendibility condition in Remark~\ref{TGextend} is satisfied
whenever $G/N$ is cyclic.
We obtain the following consequence.

\begin{cor}\label{TGexpo}
 For any finite group $G$, the exponent of $G/G'$ divides $T(G)$.
\end{cor}

\begin{proof}
 Let $n$ be the exponent of $G/G'$. Then $G$ contains a normal subgroup $N$
 such that $G/N$ is cyclic of order $n$. Thus the result follows from
 Remark~\ref{TGextend}.
\end{proof}

Under certain hypotheses, the divisibility assertion above generalizes.

%
 
\begin{rem}\label{TGpcube} 
In the situation of Proposition~\ref{TGClifford}, suppose that $G/N$ is 
abelian of order $p^3$ where $p$ is a prime. We claim that $p^2$ divides $T(G)$.

Indeed, we have again that
$$T(G) = \sum_{\phi \in \Irr(N)} T(G_\phi/N,\alpha_\phi) \cdot \phi(1).$$
If $[G:G_\phi] \geq p^2$ then $\phi$ has $[G:G_\phi]$ conjugates under $G$, 
and each of these contributes the same value to $T(G)$. 

If $[G:G_\phi] = p$ then $\phi$ has $p$ conjugates under $G$, and each
of these contributes the value $T(G_\phi/N,\alpha_\phi) \cdot \phi(1)$ to
$T(G)$. If $\phi$ 
extends to $G_\phi$ then we have $T(G_\phi/N,\alpha_\phi) = T(G_\phi/N) = 
|G_\phi/N| = p^2$.
If $\phi$ does not extend to $G_\phi$ then $\Irr(G_\phi/N,\alpha_\phi)$ 
consists of a single character of degree $p$.

Now suppose that $G_\phi = G$. If $\phi$ extends to $G$ then we 
obtain $T(G_\phi/N,\alpha_\phi) = T(G/N) = |G/N| = p^3$. 
If $\phi$ does not extend 
to $G$ then $\Irr(G/N,\alpha_\phi)$ consists of $p$ characters of degree
$p$, so that $T(G_\phi/N, \alpha_\phi) = T(G/N,\alpha_\phi) = p \cdot p = p^2$.

The claim follows.
\end{rem}

 We also have the following consequence of Remark~\ref{TGextend}:
 
\begin{rem}
 Let $N$ be a normal subgroup of a finite group $G$ such that $G/N$ is
 abelian and $\gcd(|N|,[G:N]) = 1$. Then $[G:N]$ divides $T(G)$. 
\end{rem}

Next we apply Proposition~\ref{TGcentral} to groups of prime power order.

\begin{lem}\label{pgrouplemma}
 Let $p$ be a prime, and let $G$ be a finite group of order $p^n =
 p^{2m+e}$ where $m$ and $n$ are integers, $n > 0$, and $e \in \{0,1\}$.
 Moreover, let $N$ be a minimal normal subgroup of $G$. Then
 $$T(G) \geq T(G/N) + (p-1) \cdot p^m.$$
 If equality holds then $\chi(1) = p^{m+e-1}$ for every $\chi \in 
 \Irr(G) \setminus \Irr(G/N)$; here we identify characters of $G/N$
 with their inflations to $G$.
\end{lem}

\begin{proof}
This is an easy consequence of Proposition~\ref{TGcentral} since $N
\subseteq Z(G)$ and $b(G) \leq p^{m+e-1}$ holds.
%
\end{proof}

Next we prove a lower bound for the number $T(G)$, in the case where the 
order of $G$ is a prime power. 

\begin{prop}\label{HefMacH}
 Let $G$ be a finite group of order $p^n = p^{2m+e}$ where 
 $p,n,m,e$ are as above,
 except that $n = 0$ is now allowed. Then 
 $$T(G) \geq p^{m+1} + p^{m+e} - p.$$
 If equality holds and $n \geq 3$ then $\Irr(G)$ consists of $p^2$
 linear characters, $p^2-1$ characters of degree $p^i$, for $i=1,
 \ldots, m-1$, and $p^e-1$ characters of degree $p^m$. In particular,
 we have $k(G) = p^e + (p^2-1) \cdot m$.
\end{prop}

\begin{proof}
 If $n=0$ then $m = e = 0$, $|G| = 1$ and $T(G) = 1 = p+1-p$. Suppose
 now that $n \geq 1$, and choose a minimal normal subgroup $N$ of $G$.
 Then $|G/N| = p^{n-1} = p^{2(m+e-1)+(1-e)}$. Arguing by induction,
 Lemma~\ref{pgrouplemma} above shows:
 \begin{eqnarray*}
  T(G) &\geq &T(G/N) + (p-1) \cdot p^m 
  \geq p^{m+e} + p^m - p + p^{m+1} - p^m \cr
  &= &p^{m+1} + p^{m+e} - p.
 \end{eqnarray*}
 Now suppose that $n \geq 3$ and $T(G) = p^{m+1} + p^{m+e} - p$. 
 Then we must have
 $T(G) = T(G/N) + (p-1) \cdot p^m$. Thus Lemma~\ref{pgrouplemma} implies that 
 $\Irr(G) \setminus \Irr(G/N)$ consists of $(p-1) \cdot p^{1-e}$ characters
 of degree $p^{m+e-1}$. Moreover, we also have $T(G/N) = p^{m+e} +
 p^m - p$. Thus, by induction, $\Irr(G/N)$ consists of $p^2$ linear
 characters, $p^2-1$ characters of degree $p^i$, for $i = 1, \ldots,
 m+e-2$, and $p^{1-e}-1$ characters of degree $p^{m+e-1}$. Thus the 
 total number of characters of degree $p^{m+e-1}$ in $\Irr(G)$ is
 $p^{2-e}-1$. This means that, in case $e=1$, $\Irr(G)$ contains 
 $p-1$ characters of degree $p^m$, and in case $e=0$, $\Irr(G)$
 contains $p^2-1$ characters of degree $p^{m-1}$ (and no characters
 of degree $p^m$). Hence the character degrees are as claimed.
 Furthermore, we conclude that 
 \[
   k(G) = |\Irr(G)| = p^2 + (m-1) \cdot (p^2-1) + p^e-1 = p^e + m \cdot (p^2-1).
 \]
\end{proof}

\begin{rem}\label{Poland}
 (i) The first part of Proposition~\ref{HefMacH} confirms a conjecture of 
 Heffernan and MacHale \cite{HM} (see the text after Proposition 8).
 
 (ii) Proposition~\ref{HefMacH} also shows that a finite group $G$ of order
 $p^n = p^{2m+e}$ as above with $T(G) = p^{m+1} + p^{m+e} - p$ has
 the minimal possible class number predicted by Hall's class number 
 formula \cite[Satz~V.15.2]{HuppertI}). Thus a result by Poland 
 \cite{Poland} shows
 that such a group is a $p$-group of maximal class. Moreover, its
 commutativity degree (see~\cite{Blackburn} for a definition) is 0 or 1.
 This implies that its order is at
 most $p^{p+2}$ for $p \leq 7$, and at most $p^{p+1}$ for $p>7$
 (cf. \cite{FAS}). 
 
 (iii) There are finite groups $G,H$ of the same prime power order such
 that $k(G) < k(H)$, but $T(G) > T(H)$. For example, this happens for 
 groups of order $2^6$ and $3^7$.
\end{rem}

\begin{example} 
 For a prime $p$ and a nonnegative integer $n = 2m+e$ as above
 we denote by $T(p^n)$ the minimum of the values $T(G)$ where $G$ runs 
 through the groups of order $p^n$. Thus $T(p^n) \geq p^{m+1} + p^{m+e} -p$
 by Proposition~\ref{HefMacH}. Computations in \cite{HM} show that 
 $T(p^n) = p^{m+1} + p^{m+e}-p$ whenever $p=2$ and $n \leq 4$, 
 or $p=3$ and $n \leq 5$, or $p\geq 5$  and $n \leq 6$.
 On the other hand, computations with \GAP~\cite{GAP} show that 
 $T(p^n) = p^{m+1} + p^{m+e} - p + p \cdot (p-1)^2$ in the following cases:
 $$p^n \in \{2^5,2^6,2^7,3^6,3^7\},$$
 and that $T(2^8) = 56 > 48$. The computations in Section~6 will show
 that $T(2^9) = 80$ and $T(2^{10}) = 112$. 
\end{example}

Our next result can be viewed as an analog of P.~Hall's class number formula
for groups of prime power order (cf. \cite[Satz V.15.2]{HuppertI}).

\begin{prop}\label{TGHall}
 Let $G$ be a finite group of order $p^n = p^{2m+e}$ where
 $p,n,m,e$ are as above. Then there exists a nonnegative integer $t$
 such that 
 $$T(G) = p^{m+1} + p^{m+e} - p + t \cdot p \cdot (p-1)^2.$$
\end{prop}

\begin{proof}
 Algebraic conjugation shows that, for every integer $d>1$, the number of 
 irreducible characters of $G$ of degree $d$ is divisible by $p-1$.
 We conclude that 
 \[
  |G| - T(G) = \sum_{\chi \in \Irr(G)} \left(\chi(1)^2 - \chi(1)\right) \\
       = \sum_{\chi \in \Irr(G) \setminus \Lin(G)}
                  \chi(1) \cdot \left(\chi(1) - 1\right)
 \]
is divisible by $p \cdot (p-1)^2$.
Moreover, we have $p^k \equiv k \cdot (p-1) + 1 \pmod{(p-1)^2}$
for every nonnegative integer $k$. This implies:
\begin{eqnarray*}
 p^{m+1} + p^{m+e} - p & \equiv & (m+1) \cdot (p-1) + 1 + (m+e) \cdot (p-1) +1 - p \\
 & = & (2m+e) \cdot (p-1) + 1 \equiv p^n \pmod{(p-1)^2}.
\end{eqnarray*}
Thus we obtain
$$T(G) \equiv |G| \equiv p^{m+1} + p^{m+e} - p \pmod{p(p-1)^2},$$
and the result follows from Proposition~\ref{HefMacH} above.
\end{proof}

\begin{rem}
 (i) For a prime $p$ and a nonnegative integer $n$, let $k(p^n)$ denote 
 the minimum of all $k(H)$ where $H$ runs through the groups of order 
 $p^n$.
 
 Then there are examples of groups $G$ of order $p^n$ such that
 $k(G) = k(p^n)$ but $T(G) > T(p^n)$.
 For example, this happens for $p^n = 2^5$ and $p^n = 3^6$.
 
 Similarly, there are examples of groups $G$ of order $p^n$ such that
 $T(G) = T(p^n)$, but $k(G) > k(p^n)$. For example, this happens for 
 $p^n = 2^8$ and $p^n = 2^9$. 

 (ii) Let $G$ be a group of order $p^n > p^{p+2}$. Then 
 Proposition~\ref{TGHall} and Remark~\ref{Poland} imply that 
 $$T(G) \geq p^{m+1} + p^{m+e} + p^3 -2p^2.$$
\end{rem}

\begin{rem} 
 As in Proposition~\ref{TGHall}, we write
 $$T(G) = p^{m+1} + p^{m+e} - p + t \cdot p \cdot (p-1)^2.$$
 On the other hand, Hall's class number formula \cite[Satz~V.15.2]{HuppertI} 
 gives
 $$k(G) = p^e + (p^2-1) \cdot m + (p^2-1) \cdot (p-1) \cdot s$$
 where $s$ is a nonnegative integer sometimes called the abundance of 
 $G$ (cf. \cite{FAS}, for example). Inserting these two expressions 
 into the congruence
 \begin{eqnarray*}
  |G| - 2 T(G) + k(G) &= 
  &\sum_{\chi \in \Irr(G)} (\chi(1)^2 - 2 \chi(1) + 1) \cr
  &= &\sum_{\chi \in \Irr(G) \setminus \Lin(G)} (\chi(1) - 1)^2 \cr
  &\equiv &0 \pmod{(p-1)^3}
 \end{eqnarray*}
we obtain that $2s \equiv 2t \pmod{p-1}$.
\end{rem}

We end this chapter with a special result for the prime $2$.

\begin{lem}\label{TGtwo}
 Let $G$ be a finite group of order $2^n \geq 8$ where $n$ is an integer.
 If $G$ does not have maximal class then $T(G)$ is divisible by $4$.
\end{lem}

\begin{proof}
 If $G$ does not have maximal class then $8$ divides $[G:G']$
 (cf.~\cite[Satz~III.11.9]{HuppertI}, for example).
 Thus the result follows from Remark~\ref{TGpcube}.
\end{proof}
\vspace{1cm}


\section{Some classification results}\label{sect:classifications}

In this section we prove some results which classify certain groups $G$
in terms of the invariant $T(G)$. These results were obtained in connection
with the computation of a database of all finite groups $G$ with $T(G) \leq 
100$ (cf. Section 6).

\begin{lem}\label{TGprime}
 Let $G$ be a finite group and $p$ a prime. Then $T(G)$ is odd if and only
 if $|G|$ is odd. Moreover, we have $T(G) = p$ if and only if $|G| = p$.
\end{lem}

\begin{proof}
 The first assertion holds since 
 $$|G| = \sum_{\chi\in \Irr(G)} \chi(1)^2 
 \equiv \sum_{\chi \in \Irr(G)} \chi(1) = T(G) \pmod{2}.$$
 If $|G| = p$ then clearly $T(G) = p$ holds. Conversely, suppose that 
 $T(G) = p$. If $T(G) = 2$ then obviously $|G| = 2$. 
 Thus we may assume that $p$ is odd. Then, by the first assertion,  $|G|$ is
 odd. By the Feit-Thompson theorem, $G$ is solvable. Let $N$ be a maximal 
 normal subgroup of $G$. Since $G/N$ is cyclic, Corollary~\ref{TGexpo}
 implies that $[G:N] = p$, so that $T(G/N) = p = T(G)$. We conclude that
 $N=1$ and $|G| = p$. 
\end{proof}

Our next aim is a classification of the finite groups $G$ with $T(G) = p^2$
where $p$ is a prime. We start with the following useful elementary
observation from \cite{BM}.

\begin{lem}\label{BerMa}
 Let $H$ be a proper subgroup of a finite group $G$. Then $T(H) < T(G)$.
\end{lem}

We can now prove the announced result.

\begin{prop}\label{TGprimesquare}
 Let $p$ be a prime, and let $G$ be a finite group with $T(G) = p^2$.
 Then one of the following holds:
\begin{enumerate}
 \item[(a)]
   $G$ is abelian of order $p^2$.
 \item[(b)]
   $G$ is a Frobenius group of order $pq^n$ where $q$ is a prime and $n$
   is a positive integer. Moreover, the Frobenius kernel of $G$ is a
   minimal normal subgroup of $G$ of order $q^n = p^2 - p + 1$.
\end{enumerate}
\end{prop}

\begin{proof}
  We can assume that $G$ is nonabelian. For $p=2$, we have $T(G) = 4$. 
  Thus Corollary~\ref{TGinequalities} and Remark~\ref{Berko} imply that
  $|G| < T(G)b(G) \leq 4 \cdot 2 = 8$. Thus $G$ is isomorphic to $S_3$, and the proposition is proved in this case.
	
  In the following, we therefore assume that $p$ is odd. Then $G$ has odd order, by Lemma~\ref{TGprime}. Thus the Feit-Thompson theorem implies 
  that $G$ is solvable; in particular, we have $1 < G' < G$. By 
  Corollary~\ref{TGexpo}, $p$ is the only prime divisor	of $|G/G'|$. 
  Thus $G/G'$ is a $p$-group, and $|G/G'| = T(G/G') < T(G) = p^2$, 
  so that $|G/G'| = p$. Since $p^2 = T(G) \geq [G:G'] + b(G) = p + b(G)$,
  Corollary~\ref{TGinequalities} shows that $|G| < T(G)b(G) \leq 
  p^2(p^2-p) = p^4 - p^3$. This implies that $|G'| < p^3 - p^2$. 	
  Since $G$ is nonabelian and $|G/G'| = p$, $G$ cannot be a $p$-group. 
  Let $P$ be a Sylow $p$-subgroup of $G$. Since $P < G$, 
  Lemma~\ref{BerMa}
  implies that $p^2 = T(G) > T(P) \geq T(P/P') = |P/P'|$. We conclude
  that $|P/P'| = p$, so that $P' = 1$ and $|P| = p = |G/G'|$. Hence 
  $G'$ is a $p'$-group. 
	
  Standard results on coprime actions imply that 
  $$G'/G'' = C_{G'/G''}(G) \times [G'/G'',G/G'']$$
  where $[G'/G'',G/G''] = [G',G]/G''$. Then $G/[G',G]$ has a central subgroup
  with cyclic factor group, so that $G/[G',G]$ is abelian. But then we 
  conclude that	$G' \subseteq [G',G] \subseteq G'$, so that $G' = [G',G]$ 
  and therefore $C_{G'/G''}(G) = 1$. This means that $G/G'$ acts 
  fixed-point-freely on $G'/G''$. Hence $G/G''$ is a Frobenius group with 
  kernel $G'/G''$. Thus Remark~\ref{TGinequality}~(iii) shows that
  $p^2 = T(G) \geq T(G/G'') = p + |G'/G''| - 1$. This implies: $|G'/G''| \leq 
  p^2-p+1$.
	
  Since every nontrivial orbit of $G/G'$ on $G'/G''$ has length $p$, we also 
  have 	$p \; \big| \; |G'/G''| - 1$. This implies that $|G'/G''| \geq
  p+1$ and $|G''| \leq p^2 - 2p + 1 = (p-1)^2$.
	
  Assume that $G'/G''$ is not a chief factor of $G$. Then there exists a 
  normal subgroup $N$ of $G$ such that $G'' < N < G'$. Since $G/G'$ acts 
  nontrivially on $G'/N$ and on $N/G''$, we have $|G'/N| > p$ and 
  $|N/G''| > p$ which leads	to the contradiction $|G'/G''| > p^2$. 

  This means that $G'/G''$ is in fact a chief factor of $G$; in particular, 
  $G'/G''$ is elementary abelian of order $q^n$ where $q$ is a prime and $n$
  is a positive integer. 
  
  If $G'' = 1$ then $G$ is a Frobenius group of order $pq^n$ with 
  Frobenius kernel $G'$. Moreover, $G'$ is elementary abelian of order
  $q^n$ and a minimal normal subgroup of $G$. Since $p^2 = p + q^n - 1$
  the result follows in this case. 
	
  From now on, we assume that $G'' > 1$ and choose a chief factor $G''/M$ 
  of $G$. We can consider $V := G''/M$ as a simple $G/G''$-module. By
  Clifford theory, its restriction to $G'/G''$ is a semisimple 
  $G'/G''$-module. Let $U$ be a simple $G'/G''$-submodule of $V$. 
  
  If $U$ is not $G/G''$-stable then $V$ is the direct sum of $p$ simple 
  $G'/G''$-modules, all conjugate to $U$. But this leads to the 
  contradiction 
  $$(p-1)^2 \geq |G''| \geq	|G''/M| = |V| = |U|^p \geq 3^p.$$
	
  This shows that $U$ is $G/G''$-stable. By the Dade-Isaacs theorem
  (cf. \cite{Dade}), $U$ extends from the Hall $p'$-subgroup $G'/G''$ of $G/G''$ 
  to $G/G''$. By a result of Glasby and Kov\'acs in \cite{GK}, all 
  composition factors of the induced module $U^{G/G''}$ have the same
  dimension. Thus $V$ restricts to $U$, so that $V$ is simple as a 
  $G'/G''$-module. This means that $G''/M$ is also a chief factor of $G'$. 
  
  Obviously, the normal subgroup $C := C_G(V)$ of $G$ contains $G''$. 
  If $C = G''$ then $G'/G''$ acts faithfully and irreducibly on
  $V = G''/M$. Thus it also acts faithfully and irreducibly on $\Irr(G''/M)$. 
  Since $G'/G''$ is an abelian $q$-group it has a regular orbit on 
  $\Irr(G''/M)$ (cf.~\cite[Theorem~4.8]{MW}, for example).
  Let $\lambda$ be an element in this orbit, and let $T/G''$ be its 
  stabilizer in $G/G''$. Then $T/G'' \cap G'/G'' = 1$, i.e.\ $T \cap G' =
  G''$. Thus $T/G''$ is isomorphic to $TG'/G'$; in particular, $T/G''$ has
  order	$1$ or $p$.
	
  If $|T/G''| = 1$ then the $G/G''$-orbit of $\lambda$ contains $|G/G''| = 
  pq^n > p^2$ elements which is impossible. Thus we must have $|T/G''| =
  p$. Then there are $p$ distinct characters $\lambda_1,\ldots,
  \lambda_p$ of $T$ extending $\lambda$. By Clifford theory, $\lambda_1^G,
  \ldots,\lambda_p^G$ are distinct irreducible characters of $G$,
  so that $T(G) \geq p \lambda_1^G(1) = pq^n > p^2$, and we have again a 
  contradiction. 
	
  We conclude that $G'' < C$. Thus the normal subgroup structure of $G/G''$ 
  implies that $G' \subseteq C$. We conclude that $G''/M \subseteq Z(G'/M)$;
  in particular, $G'/M$ is nilpotent.
	
  If $G'/M$ is not a $q$-group then $G'/M$ is isomorphic to 
  $G'/G'' \times G''/M$; in particular, $G'/M$ is abelian. Since $M < G''$
  this is a contradiction.
	
  Thus $G'/M$ is in fact a $q$-group. Since $G''/M$ is a chief factor of 
  $G'/M$ we obtain that $|G''/M| = q$. Note that $G'/M$ is not abelian 
  since $M < G''$. Thus we have $Z(G'/M) < G'/M$. Since $G'/G''$ is a chief
  factor of $G$, this implies that $Z(G'/M) = G''/M = (G'/M)'$. Since 
  $G'/G''$ is elementary abelian we also have $\Phi(G'/M) \subseteq G''/M
  \subseteq \Phi(G'/M)$. This shows that $G'/M$ is an extraspecial $q$-group;
  in particular, $n$ is even. We set $m := n/2$. Since $p$ divides 
  $q^n - 1 = (q^m + 1)(q^m - 1)$ we see that $p$ divides $q^m+1$ or $q^m-1$. Since $p$ and $q$ are odd this implies $p < q^m$, and we obtain the
  contradiction $p^2 < q^{2m} = |G'/G''|$.
\end{proof}

\begin{example}
For $p=2,3,7$, there is a Frobenius group $G$ of order $6,21,301$, respectively,
with $T(G) = p^2$. For $p=5,11$, such a Frobenius group does not exist.
\end{example}


In the following we will derive a somewhat strange characterization of the 
groups $\Qd(p)$ where $p$ is a prime. Recall that $\Qd(p)$ denotes the 
semidirect product of $\SL(2,p)$ and an elementary abelian group of order
$p^2$, where the action is canonical. In \cite{HuppertII} the group $\Qd(p)$
is denoted by $\SA(2,p)$. Note that $\Qd(2)$ is isomorphic to the symmetric
group $S_4$.

We start with the result for $p=2$; this is easier than the case of odd 
primes and also a little different.

\begin{prop}\label{Qd2}
Let $G$ be a finite group with a Sylow $2$-subgroup $P$ of order $8$,
and suppose that $T(G) \leq 12$. Then $P$ is normal in $G$, or $G$ is 
isomorphic to $S_4$. (Note that $T(S_4) = 1+1+2+3+3 = 10$.)
\end{prop}

\begin{proof}
Let $G$ be a counterexample. Then we have $|G| \in \{24, 40,
56, 72, 88, \ldots\}$. This implies $k(G) \geq 5$, so that $b(G) \leq T(G)
- k(G) + 1 \leq 8$ and $|G| < T(G) b(G) \leq 96$. This shows that $|G| \leq
88$. Let $q$ be the unique odd prime divisor of $|G|$, and let $Q$ be a 
Sylow $q$-subgroup of $G$. Then $G = PQ$, and $|Q| \in \{3, 5, 7, 9, 11
\}$.

Suppose first that $|Q| = q$. Then $Q$ is normal in $G$
(cf.~\cite[Theorems~1.32 and 1.33]{IsaacsGT}, for 
example). Since $G$ is not nilpotent, $C := C_G(Q) = C_P(Q)Q$ is a 
proper subgroup of $G$ and abelian. Moreover, $G/C$ is isomorphic to a 
subgroup of $\mathrm{Aut}(Q)$; in particular, $|G/C|$ divides $q-1$. Now
Remark~\ref{TGinequality} implies that
$$12 \geq T(G) \geq T(G/C) + T(C) - 1 = |G/C| + |C| - 1.$$

If $|G/C| = 2$ then $|C| = 4q \geq 12$, and we have a contradiction.

If $|G/C| = 4$ then $q \geq 5$ and $|C| = 2q \geq 10$, and we 
have a contradiction again.

If $|G/C| = 8$ then $|C| = q \geq 17$ which is also impossible.

It remains to consider the case $|Q| = 9$. Then we have $|G|
= 72$ and $k(G) \geq 6$, so that $b(G) \leq T(G) - k(G) + 1 \leq 7$. Since
$b(G)$ divides $|G|$ we even obtain $b(G) \leq 6$. This gives the 
contradiction $|G| < T(G) b(G) \leq 72$.
\end{proof}

Now we turn to the case of odd primes. We start by collecting a number
of useful results.

\begin{lem}\label{Qd(p)lemma}
Let $S$ be a nonabelian finite simple group. Then the following holds:

\noindent (i) $|S| > 2 b(S)^2$;

\noindent (ii) $T(S) > 2b(S)$;

\noindent (iii) $b(S) \geq |A|$ for every abelian subgroup $A$ of $S$;

\noindent (iv) $|S| > |B|^2$ for every nilpotent subgroup $B$ of $S$.
\end{lem}

\begin{proof}
(i) This is \cite[Theorem~1.4]{HLSF}.

(ii) Since $|S| < T(S) b(S)$ the assertion follows from (i).

(iii) This is \cite[Theorem~1.1]{HY}.

(iv) This is \cite[Theorem~2.2]{Vd}.
\end{proof}

Now we can prove an analog of Proposition~\ref{Qd2} for odd primes.

\begin{prop}\label{Qdp}
Let $p$ be an odd prime, and let $P$ be a Sylow $p$-subgroup
of a finite group $G$. Suppose that $|P| = p^3$ and $T(G) \leq p^3 + p^2$. 
Then either $P$ is normal in $G$, or $G$ is isomorphic to ${\Qd}(p)$. 
(It is easy to see that $T({\Qd}(p)) = p^2 + p + (p^2-1)p = p^3 + p^2$ 
for $p>2$.)
\end{prop}

\begin{proof}
Let $G$ be a minimal counterexample, and let $N$ be a minimal normal 
subgroup of $G$. Then $C := C_G(N)$ is also normal in $G$.

If $H$ is a proper subgroup of $G$ containing $P$ then we have $T(H) < T(G) 
\leq p^3 + p^2$, by Lemma~\ref{BerMa}. Thus, by the minimality of $G$, 
$H \subseteq N_G(P)$ holds. In particular, 
$N_G(P)$ is the only maximal subgroup of $G$ containing $P$.

This argument shows that $G = O^{p'}(G)$; for otherwise we would have
$O^{p'}(G) \subseteq N_G(P)$. Thus $P$ would be the only Sylow $p$-subgroup 
of $O^{p'}(G)$ and would therefore be normal in $G$.

Suppose that $N$ is a $p'$-group. Then the Sylow $p$-subgroup $PN/N$ of 
$G/N$ has order $p^3$. Since $T(G/N) < T(G) \leq p^3 + p^2$, the minimality 
of $G$ implies that $PN/N$ is normal in $G/N$. Hence $PN$ is normal in $G$.

If $PN$ is a proper subgroup of $G$ then $P$ is normal in $PN$. Hence 
$P = O_p(PN)$ is normal in $G$, and we have a contradiction. This shows that
$PN = G$.

If $N$ is nonabelian then, for any prime $r$, $N$ has a $P$-invariant Sylow 
$r$-subgroup $R$. Since $PR$ is a proper subgroup of $G$, $P$ is normal in 
$PR$. This implies that $R \subseteq C_G(P)$. Since $r$ was arbitrary we 
even have $N \subseteq C_G(P)$ which is impossible.

Thus $N$ must be elementary abelian. Hence $C = C_P(N) \times N$ is a proper 
normal subgroup of $G$ and abelian. Since $p$ is odd, $P/C_P(N)$ has a 
regular orbit on $N$ (cf.~\cite[Theorem~8]{HuMa}, for example). This 
implies: $|N| \geq |P/C_P(N)| + 1$ and $T(C) = |C| \geq |P| + |C_P(N)|$. 
Moreover, Remark~\ref{TGinequality} yields:
$$T(G) \geq T(G/C) + T(C) - 1 \geq T(P/C_P(N)) + |P| + |C_P(N)| -1.$$
If $C_P(N) = 1$ then we obtain the contradiction 
$T(G) \geq T(P) + p^3 \geq 2p^2 - p + p^3 > p^3 + p^2$, 
and if $C_P(N) > 1$ then we reach the contradiction 
$$T(G) \geq |P/C_P(N)| + p^3 + |C_P(N)| - 1 = p^3 + p^2 + p - 1 
> p^3 + p^2.$$
This finishes the case where $N$ is a $p'$-group. From now on, we can 
therefore assume that $O_{p'}(G) = 1$. 
Next we consider $E(G)$, the layer of $G$, and 
$F^\ast(G)$, the generalized Fitting subgroup of $G$.

If $G$ has a component $S$ with $Z(S) > 1$ then $p$ divides $|Z(S)|$, and 
therefore $p$ divides $|\mathrm{H}^2(S/Z(S), \mathbb{C}^\times)|$. 
Hence $p^2$ divides $|S/Z(S)|$, so that
$p^3$ divides $|S|$. Thus $S$ contains $P$. Since $P$ is not normal in $S$, 
we must have $S = G$, and $N = Z(G)$ has order $p$. 
Remark~\ref{TGpractice} shows:
$$|G| \leq T(G)^2/|N| \leq p^3(p+1)^2,$$
so that $|G/N| \leq p^2(p+1)^2$. On the other hand, parts (i) and (iii) of 
Lemma~\ref{Qd(p)lemma} imply that $|G/N| > 2p^4$. This is a contradiction.

We have therefore proved that every component of $G$ is simple. Thus we can 
write $E(G) = S_1 \times \ldots \times S_k$ with nonabelian finite simple 
groups $S_1,\ldots,S_k$. Since $p$ divides $|S_i|$ for $i=1,\ldots,k$, we 
have $k \leq 3$. We also set $S_0 := F(G) = O_p(G)$. Then we have 
$|S_0| \leq p^2$ and $F^\ast(G) = S_0 \times S_1 \times \ldots \times S_k$. 

Suppose next that $p^3$ divides $|S_i|$ for some $i \in \{1,\ldots,k\}$. 
Since $S_i$ does not have a normal Sylow $p$-subgroup, this implies that 
$G = S_i$ is simple. 
But now part (iv) of Lemma~\ref{Qd(p)lemma} shows that $|G| > p^6$. 
On the other hand, part (ii) of Lemma~\ref{Qd(p)lemma} implies that 
$T(G) > 2b(G)$, so that $|G| < T(G) b(G) < T(G)^2/2 \leq p^4(p+1)^2/2$, 
and we have a contradiction. 

This shows that $p^3$ does not divide $|S_i|$, for $i = 0,\ldots,k$; 
in particular, the Sylow $p$-subgroups of $F^\ast(G)$ are abelian. 
Lemma~\ref{Qd(p)lemma} above implies:
\begin{eqnarray*}
 T(F^\ast(G)) b (F^\ast(G)) &\geq &|F^\ast(G)| 
           = |S_0| \cdot |S_1| \cdots |S_k| \cr
           &\geq &|S_0| \cdot 2b(S_1)^2 \cdots 2 b(S_k)^2 \cr
           &\geq &2^k |S_0| \cdot | P \cap S_1| \cdots |P \cap S_k| b(S_1 
           \times \ldots \times S_k) \cr
           &= &2^k |P \cap F^\ast(G)| b(F^\ast(G)).
\end{eqnarray*}
Thus $2^k|P \cap F^\ast(G)| \leq T(F^\ast(G)) \leq T(G) \leq p^3 + p^2$.

If $P \subseteq F^\ast(G)$ then we must have $k>0$ since $|S_0| \leq p^2$. 
But then the inequality above gives the contradiction $2p^3 \leq p^3 + p^2$.

We conclude that $P$ is not contained in $F^\ast(G)$; in particular, 
we have $| P \cap F^\ast(G)| \leq p^2$. Moreover, $P$ is nonabelian, by
\cite[Corollary~1.2]{GGLN}. 
Thus $P$ has $p$-rank $2$, and $Z(P)$ is the only minimal normal subgroup of 
$P$. Hence we must have $k \leq 2$. Since $p$ is odd, $P$
normalizes $S_i$, for $i = 0,1,\ldots,k$. Thus 
$P \cap S_1,\ldots,P \cap S_k$ are 
nontrivial normal subgroups of $P$. This shows that $k \leq 1$.

Suppose that $k=1$. Then, by the same argument, we have $O_p(G) = 1$. 
Thus $F^\ast(G) = S_1$ is simple, and $G$ is almost simple. 
Moreover, we have $G = F^\ast(G) P > F^\ast(G)$. 
Thus \cite[Corollary~4]{BBGT} implies that
either $p=3$ and $G$ is isomorphic to one of ${\PSL}_2(8).3$, 
${\PSU}_3(8).3$, ${\PSU}_3(8).3^2$, or $p=5$ and $G$ is isomorphic to 
${^2}B_2(32).5$. However, 
these cases are impossible since $|G| < T(G)^2 \leq p^4(p+1)^2$.

We are left with the case $k=0$, i.e.\ $E(G) = 1$ and $F^\ast(G) = O_p(G)$. 
Then $G$ is $p$-constrained. If $O_p(G)$ is cyclic then $G/O_p(G)$ is 
isomorphic to a subgroup of the abelian group $\mathrm{Aut}(O_p(G))$. 
Since $G/O_p(G)$ does not have a normal 
Sylow $p$-subgroup, this is a contradiction. 

Hence $O_p(G)$ must be elementary abelian of order $p^2$. 
Then $G/O_p(G)$ is isomorphic to a subgroup of ${\GL}_2(p)$. 
Since $G/O_p(G) = O^{p'}(G/O_p(G))$, $G/O_p(G)$ is 
even isomorphic to a subgroup of ${\SL}_2(p)$. 
Since $G/O_p(G)$ does not have a 
normal Sylow $p$-subgroup,
Dickson's theorem \cite[Satz~II.8.27]{HuppertI} implies that $G/O_p(G)$ is 
isomorphic to ${\SL}_2(p)$ (cf.~\cite[Theorem~3.6.17]{SuzI}, for
example).
Now \cite[Lemma~IX.7.9]{HuppertII} shows that $G$ is isomorphic to ${\Qd}(p)$;
in particular, $G$ is not a counterexample.
\end{proof}


\section{Groups of small total degree}%
\label{sect:algorithm}

The aim of this section is to classify the finite groups $G$
with $T(G) \leq 100$. The main ideas of an inductive algorithm
to obtain such a classification are as follows.

Let
$G$ be a finite group with $T(G) = n$. Then one of the following occurs:

\begin{itemize}
 \item 
 If $n$ is a prime number then $G$ is a cyclic group of order $n$, by
 Lemma~\ref{TGprime}. Thus w.l.o.g. we may assume that $n$ is composite.

 \item
 If $G$ is a nonabelian finite simple group then $|G| < n^2$ holds, by
 Corollary~\ref{TGinequalities} and Remark~\ref{Berko}. 
 (The finite simple groups of small order are known,
 even if we do not want to use the classification of finite simple groups,
 and their total degrees are easily computed.) This leads to the 
 list of nonabelian finite simple groups of total degree up to 100
 shown in the left part of Table~\ref{table_T_almost_simple}.

\begin{table}
\centering
\caption{Almost simple groups with $T(G) \leq 100$}
\label{table_T_almost_simple}
\begin{tabular}{l|r}
 $G$ & $T(G)$ \\ \hline
 $A_5$ & $16$ \\
 $\PSL_2(7)$ & $28$ \\
 $A_6$ & $46$ \\
 $\PSL_2(8)$ & $64$ \\
 $\PSL_2(11)$ & $66$ \\
 $\PSL_2(13)$ & $92$ \\
\end{tabular}
\quad\quad
\begin{tabular}{l|r}
 $G$ & $T(G)$ \\ \hline
 $S_5$ & $26$ \\
 $\PGL_2(7) = \PSL_2(7).2$ & $50$ \\
 $M_{10} = A_6.2_3$ & $66$ \\
 $S_6$ & $76$ \\
 $\PGL_2(9) = A_6.2_2$ & $82$ \\
 $\PGammaL_2(8) = \PSL_2(8).3$ & $96$ \\
\end{tabular}
\end{table}

As mentioned in the proof of Proposition~\ref{prop:indexpr},
$T(\PSL_2(q)) = q(q+1)/2$ if $q \equiv 3 \pmod{4}$ and
$T(\PSL_2(q)) = 1 + q(q+1)/2$ if $q \equiv 1 \pmod{4}$. 
For even $q$, $T(\PSL_2(q)) = q^2$ holds.

\item 
If the Fitting subgroup of $G$ is trivial then we have
$$F^\ast(G) = E(G) = S_1 \times \cdots \times S_k$$
where $S_1,\ldots,S_k$ are nonabelian finite simple groups. Then 
Lemma~\ref{BerMa} shows that 
$$100 \geq n =T(G) \ge T(E(G)) = T(S_1) \cdot \cdots \cdot T(S_k) \ge 16^k,$$
so that $k = 1$. Thus $E(G)$ is simple, and $G$ is almost simple. 
In addition to the six simple groups mentioned above,
we also get the almost simple groups of total degree up to 100
shown in the right part of Table~\ref{table_T_almost_simple}.

\item
If the Fitting subgroup of $G$ is nontrivial then $G$ has a minimal 
normal subgroup $N$ which is elementary abelian of order $p^d$ where
$p$ is a prime and $d$ is a positive integer. Then we have $T(N) = 
|N| = p^d$ and $n = T(G) \geq T(F) + T(N) - 1$ where $F := G/N$. 
Since $T(F) < T(G)$, we can assume that $F$ has been constructed already,
by induction.
Moreover, $N$ can be considered as a simple
$F$-module. The results in the preceding sections yield further
restrictions on $F$ and $N$. The candidate groups $G$ have to be
constructed from $F$ and $N$ and then it has to be checked whether
$T(G) \leq 100$. Finally, representatives for the isomorphism classes
have to be computed.
\end{itemize}

In order to create the groups in the set $\mathfrak{T}_{n}$
of finite groups $G$ with $T(G) \leq n$,
for $n < 16^2$,
we proceed as follows.

\begin{enumerate}
\item
  Initialize the result list with the trivial group,
  the cyclic groups of prime order at most $n$,
  and the almost simple groups of total degree at most $n$.
\item
  Create a ``work list'' of nontrivial groups whose extensions by an
  irreducible
  module may yield new groups of total degree at most $n$.
  Initialize it with all those groups from the initial result list
  that are nontrivial and have total degree less than $n$,
  sorted by total degree.
\item
  As long as the ``work list'' is nonempty,
  take the first group $F$ from it, delete it from the ``work list'',
  compute the possible extensions $G$ by irreducible modules,
  as described above,
  add those to the result list that have total degree at most $n$
  and do not yet occur in the result list,
  and add those that might have extensions of total degree at most $n$
  to the ``work list'', keeping its ordering by total degree.
\item
  After finitely many steps, the ``work list'' becomes empty,
  and the result list is complete.
\end{enumerate}

Table~\ref{table_T_small} lists some results of our computations. 
The first column contains the total degree. Note that we have omitted
the trivial cases where $T(G)$ is a prime number. The second column
contains the number of groups $G$ of a fixed total degree, up to 
isomorphism. The third column contains the maximal order of a group
with fixed total degree, and the last column lists the groups of
maximal order and fixed total degree.

\begin{table}
\caption{Groups with $T(G) \leq 100$}
\label{table_T_small}
\begin{minipage}{6.5cm}
\scriptsize
\[
   \begin{array}{r|r|r|r}
   T(G) & \#G & |G_{\max}| & G_{\max}  \\ \hline
   4 &     3 &    6 &               S_3  \\
   6 &     5 &   12 &               A_4  \\
   8 &     7 &   20 &              5:4   \\
   9 &     3 &   21 &              7:3   \\
  10 &     7 &   24 &               S_4  \\
  12 &    16 &   42 &              7:6   \\
  14 &     7 &   72 &         3^2:Q_8    \\
  15 &     5 &   55 &             11:5   \\
  16 &    23 &   72 &          3^2:8     \\
  18 &    19 &   78 &             13:6   \\
  20 &    35 &  110 &            11:10   \\
  21 &     2 &   57 &             19:3   \\
  22 &    13 &   72 & (S_3 \times S_3):2,\\
     &       &      & (3 \times A_4):2,  \\
     &       &      &      (2^2:9):2     \\
  24 &    55 &  156 &            13:12   \\
  25 &     2 &   25 &              5^2,  \\
     &       &      &               25   \\
  26 &     8 &  144 &         3^2:QD_{16}\\
  27 &     7 &  171 &             19:9   \\
  28 &    61 &  168 &         \PSL(3,2)  \\
  30 &    34 &  240 &           2^4:15   \\
  32 &   126 &  300 &        5^2:(3:4)   \\
  33 &     7 &  253 &            23:11   \\
  34 &    10 &  120 &             5:S_4  \\
  35 &     3 &  203 &             29:7   \\
  36 &   123 &  600 &      5^2:\SL(2,3)  \\
  38 &    13 &  216 &           3^3:Q_8  \\
  39 &     3 &  351 &           3^3:13   \\
  40 &   159 &  600 &        5^2:(3:8)   \\
  42 &    32 &  448 &      (2^3.2^3):7   \\
  44 &   114 &  506 &            23:22   \\
  45 &    18 &  465 &            31:15   \\
  46 &    19 &  360 &               A_6  \\
  48 &   400 &  600 &           5^2:24   \\
  49 &     3 &  301 &             43:7   \\
  50 &    25 &  432 & ((3^2:Q_8):3):2    \\
  51 &     7 &  243 &          (9:9):3,  \\
     &       &      &          (9:9):3,  \\
     &       &      &        (9 \times 3).3^2  \\
  52 &   257 &  702 &           3^3:26   \\
  54 &    72 &  666 &            37:18   \\
  55 &     1 &   55 &               55   \\
  56 &   799 &  812 &            29:28   \\
   \end{array}
\]
\end{minipage}
\ 
\begin{minipage}{7.5cm}
\scriptsize
\[
   \begin{array}{r|r|r|r}
   T(G) & \#G & |G_{\max}| & G_{\max}  \\ \hline
  57 &     1 &   57 &               57 \\
  58 &    25 &  784 &          7^2:Q_{16}  \\
  60 &   309 & 1176 &      7^2:\SL(2,3)  \\
  62 &    12 &  992 &           2^5:31  \\
  63 &    13 &  903 &            43:21  \\
  64 &   943 &  784 &           7^2:16  \\
  65 &     3 &  689 &            53:13  \\
  66 &    62 & 2352 & 7^2:(\SL(2,3).2)  \\
  68 &   256 &  500 &        (5^2:5):4  \\
  69 &     7 & 1081 &            47:23  \\
  70 &    43 &  610 &            61:10  \\
  72 &  1288 & 1332 &            37:36  \\
  74 &    11 &  216 &    (3^2:3):D_8,   \\
     &       &      &    (3^2:3):D_8,  \\
     &       &      &    (3^2:3):Q_8   \\
  75 &    10 &  915 &            61:15 \\
  76 &   551 & 1200 &       5^2:(24:2) \\
  77 &     3 &  737 &            67:11 \\
  78 &    56 & 2352 &      7^2:(3\times Q_{16})  \\
  80 &  2507 & 1640 &            41:40 \\
  81 &    16 &  657 &             73:9 \\
  82 &    22 &  720 & 3^2:((5 \times Q_8):2),  \\
     &       &      &               \PGL(2,9)  \\
  84 &   947 & 1806 &            43:42  \\
  85 &     2 &  405 &            3^4:5  \\
  86 &    15 &  960 &           2^4:A_5 \\
  87 &    13 & 1711 &            59:29  \\
  88 & 11384 & 1474 &            67:22  \\
  90 &   161 & 4032 &        2^6:(7:9)  \\
  91 &     4 & 1027 &            79:13  \\
  92 &   746 & 2162 &            47:46  \\
  93 &     5 & 1053 &       (3^3:13):3  \\
  94 &    35 &  648 & (2^2:((9 \times 3):3)):2,  \\
     &       &      & (2^2:(3^3:3)):2,  \\
     &       &      & (2^2:(3^3:3)):2,  \\
     &       &      & (2^2:((9 \times 3):3)):2,  \\
     &       &      & (2^2:((9 \times 3):3)):2,  \\
     &       &      & (2^2:(3^2.3^2)):2  \\
  95 &     1 &   95 &               95   \\
  96 & 18652 & 3240 &        3^4:(5:8)   \\
  98 &    32 & 1344 & 2^3:\PSL(3,2),     \\
     &       &      & 2^3.\PSL(3,2)      \\
  99 &    35 & 2211 &            67:33   \\
 100 &   660 & 1620 &           3^4:20 

   \end{array}
\]
\end{minipage}
\end{table}

Altogether there are $41\,359$ groups in $\mathfrak{T} := 
\mathfrak{T}_{100}$, up to isomorphism; 
these include, of course, all groups of order up to 100. 

Also, $\mathfrak{T}$ contains $2\,313$ groups of order $2^7 = 128$
(these are all nonabelian groups of that order)
and $32\,754$ groups of order $2^8 = 256$, but only $579$ groups of
order $2^9 = 512$ and no group of order $2^{10} = 1024$. Similarly,
$\mathfrak{T}$ contains 45 groups of order $3^5 = 243$, but no groups
of order $3^6 = 729$. It also contains the two nonabelian groups of
the orders $5^3 = 125$ and $7^3 = 343$, but
no groups of order $5^4 = 625$ or order $7^4 = 2401$.
For primes $p > 7$, the $p$-groups in $\mathfrak{T}$ are abelian. 

It is perhaps of interest that the groups of order 512 in $\mathfrak{T}$
can be generated by at most 4 elements. There are 36 groups of order 512
with total degree $T(512) = 80$; they can be generated by at most 3
elements and have nilpotency class 5 or 6.

Only 199 groups of odd order are in $\mathfrak{T}$. Of these,
198 have a Sylow tower. The only group of odd order in $\mathfrak{T}$ 
which does not have a Sylow tower is the 2-complement of
$\AGammaL(1,27)$, which has order $3^4 \cdot 13$ and total degree 93. 

There are 20 groups in $\mathfrak{T}$ whose order is divisible
by $5^3$. According to our earlier results, they all have a normal 
Sylow $5$-subgroup of order $5^3$.
Similarly, there are 5 groups in $\mathfrak{T}$
whose order is divisible by $7^3$, and they all
have a normal Sylow $7$-subgroup of order $7^3$. 

Classifying the groups $G$ in $\mathfrak{T}$ currently takes
about three hours of CPU time.
With additional data for example concerning isomorphism tests for
groups of order $512$, the classification for larger total degrees
becomes feasible.


\begin{center}
 {\bf Acknowledgements}
\end{center}

\noindent
The first author gratefully acknowledges support by the German Research
Foundation (DFG)
-- Project-ID 286237555 --
within the SFB-TRR 195 {\it Symbolic Tools in Mathematics
and their Applications}.
The authors are also grateful to Jay Taylor for pointing out the reference
\cite{AM}.



\vspace*{2cm}


\noindent
T. Breuer,
Lehrstuhl f\"ur Algebra und Zahlentheorie,
RWTH Aachen University,
Pontdriesch 14-16,
D-52062 Aachen, Germany, \\
e-mail: \texttt{sam@math.rwth-aachen.de} \\
\\
L. H\'ethelyi,
Department of Algebra,
Budapest University of Technology and Economics,
H-1111 Budapest,
M\H uegyetem rkp. 3-9,
Hungary,\\
e-mail: \texttt{lhethelyi@gmail.com} \\
\\
B. K\"ulshammer,
Institut f\"ur Mathematik,
Friedrich-Schiller-Universit\"at,
D-07737 Jena, Germany, \\
e-mail: \texttt{kuelshammer@uni-jena.de} \\
\\
M. Sz\H{o}ke,
Institute of Mathematics, \\
John von Neumann Faculty of Informatics, \\
\'Obuda University,
H-1034 Budapest,
B\'ecsi \'ut 96/B,
Hungary,\\
e-mail: \texttt{szoke.magdolna@nik.uni-obuda.hu}

\end{document}